\documentclass[reqno,12pt]{amsart}
\usepackage[latin1]{inputenc}
\usepackage{amsmath,amssymb,amsthm,ascmac}
\usepackage[mathscr,mathcal]{eucal}
\usepackage{latexsym}
\usepackage{enumerate}
\usepackage[dvips]{graphicx}
\usepackage{mathrsfs}
\usepackage{here}
\usepackage{comment}
\usepackage{mathtools}
\usepackage{nccmath}
\usepackage{amsfonts}
\usepackage[margin=35truemm]{geometry}
\usepackage{xcolor}
\usepackage{pdfpages}
\usepackage{caption}

\allowdisplaybreaks

\theoremstyle{definition}
\newtheorem{thm}{Theorem}[section]

\newtheorem{defi}[thm]{Definition}
\newtheorem{cor}[thm]{Corollary}
\newtheorem{lem}[thm]{Lemma}
\newtheorem{remark}[thm]{Remark}

\newtheorem{ex}[thm]{Example}

\def\R{\mathbb R}
\def\Z{\mathbb Z}
\def\N{\mathbb N}

\def\e{\varepsilon}
\def\coloneq{\mathrel{\mathop:}=}
\def\k{\textcolor{black}{\kappa}}

\renewcommand{\theequation}{\arabic{section}.\arabic{equation}}
\makeatletter
\@addtoreset{equation}{section}

\title[Weak Gibbs measures for $(\k/\beta, \beta)$-shifts]{\large Weak Gibbs measures for the natural extension of $(\k/\beta, \beta)$-shifts\\}
\author[M. YAMASHITA]{MIKI YAMASHITA}
\date{}

\begin{document}

\maketitle

\pagestyle{headings}
\thispagestyle{empty}

\def\bx{\boldsymbol{x}}
\def\by{\boldsymbol{y}}
\def\bz{\boldsymbol{z}}
\def\bp{\boldsymbol{p}}
\def\bq{\boldsymbol{q}}
\def\bj{\boldsymbol{j}}

\def\a{\alpha} 
\def\b{\beta} 
\def\A{\mathcal{A}} 
\def\s{\sigma}
\def\[{\lbrack} 
\def\]{\rbrack}
\def\Xab{X^{\alpha, \beta}} 
\def\Sab{\Sigma^{\alpha, \beta}} 
\def\Lab{\mathcal{L}^{\alpha, \beta}} 
\def\Tab{T_{\alpha, \beta}} 
\def\Pa{\mathcal{P}^a} 
\def\Pb{\mathcal{P}^b} 
\def\P{\mathcal{P}} 
\def\pth{\textsf{pth}} 
\def\vtx{\textsf{vtx}} 
\def\Zab{z^{\a, \b}} 
\def\M{2\Zab(b_1)+7} 
\def\Gab{\mathcal{G}^{\a, \b}} 
\def\Vab{\mathcal{V}^{\a, \b}}
\def\Eab{\mathcal{E}^{\a, \b}}
\def\Ea{\textbf{E1}\,\,}
\def\Eb{\textbf{E2}\,\,}
\def\Ec{\textbf{E3}\,\,}
\def\Ed{\textbf{E4}\,\,}
\def\B{\mathcal{B}}

\noindent\textit{Abstract.} 
In this paper we consider the weak Gibbs measures for $(\a, \b)$-shifts. 
In the case of $\a=0$, Pfister and Sullivan have given a necessary and sufficient condition on $\b$ such that any equilibrium measure for a function of bounded total oscillations is a weak Gibbs measure in the natural extension of a $\b$-shift \cite{PS1}. 
So it is natural to ask what happens when $\a>0$. 
However, their proof cannot be applied to general $(\a, \b)$-shifts in a similar way. 
In this paper we consider the case of $\a=\textcolor{black}{\k}/\b$ and give a criterion for the weak Gibbs property of equilibrium measures for $(\textcolor{black}{\k}/\b, \b)$-shifts.\\


\noindent\textit{Key words.} $(\a, \b)$-shifts, symbolic dynamics, equilibrium measures, weak Gibbs measures. 

\section{introduction}
In this paper let us consider the equilibrium measures for the natural extension of $(\a, \b)$-shifts. 
Let $0\le\a<1 , \b >1$ fixed. 
For $t\in\R$, let $\lfloor t\rfloor\coloneq\max\{i\in\Z:i\le t\}$.
The map $\Tab:[0,1)\to[0,1)$ is defined by
\begin{align}\label{Tab}
\Tab(x)=\b x+\a -\lfloor\b x+\a\rfloor=\b x+\a\,(\!\!\!\!\!\mod 1).
\end{align}
Such maps are called {\it$(\a, \b)$-transformations}. 
We define the $(\a, \b)$-shift $\Xab$ by the closure of all $(\a, \b)$-expansions obtained by $\Tab$ 
and denote as $\Sab$ the natural extension of $\Xab$.
For more than half a century dynamical properties of $\Xab$ and $\Sab$ have been investigated in many works \cite{CLR,FP,RS,Ho2,CY,Re,Pa1,Pa2}.

Haydn and Ruelle \cite{HR} proved that if an expansive homeomorphism satisfies specification, then the equivalence of equilibrium measures and Gibbs measures holds under a certain condition of potential functions.
However it is known that for $\a\in[0, 1)$ the set of all $\b>1$ for which $\Xab$ satisfies the specification has Lebesgue measure $0$ \cite{Sc, Bu}.

Yuri introduced the notion of weak Gibbs measures for piecewise $C^0$-invertible systems with finite range structure \cite{Yu1, Yu2, Yu3}. 
Later several versions of weak Gibbs measures have been defined for other dynamical systems \cite{BV, PS2, PS3}.
Pfister and Sullivan showed that if $L$ is a lattice and $X\subset A^L$ is a shift space satisfying the decoupling condition and $\varphi$ has bounded total oscillations, then any equilibrium measure for $\varphi$ is a weak Gibbs measure for $\psi=\varphi-p(\varphi)$ \cite{PS3}. Carapezza et al. gave a characterization of specification for $\Xab$ \cite{CLR}. Using this characterization we can check that the decoupling condition implies the specification for $\Sab$.


Recently, Pfister and Sullivan have given a necessary and sufficient condition on $\b$ \textcolor{black}{for the case $\a=0$} such that any equilibrium measure for a function of bounded total oscillations is a weak Gibbs measure in the natural extension of a $\b$-shift \cite{PS1}. 
So it is natural to ask what happens when $\a>0$. 

Let $a=(a_i), b=(b_i)\in\N_0^\N$ be the $(\a, \b)$-expansions of $0$ and $1$, respectively. When $\a=0$ we have $a_i=0$ for $i\in\N$, and if an admissible word $w=w_1\cdots w_n0\cdots0$ satisfies $w_n\neq0$, then $\widehat{w}=w_1\cdots w_{n-1}(w_n-1)0\cdots0$ is also admissible and it can be followed by any words.
Moreover the number of the inverse image of the transformation $w\mapsto\widehat{w}$ is bounded from above (\cite[Lemma 2.2]{PS1}).
This property plays an essential role in the proof of Pfister and Sullivan.
In the general case, it is known that $a$ has the following structure;
$$a=(a_1,\cdots, a_{n_1}, b_1,\cdots, b_{m_1}, a_{n_1+m_1+1},\cdots, a_{n_2}, b_1,\cdots, b_{m_2},\cdots),$$
for some sequence $\{n_i\}$ and $\{m_i\}$ satisfying $n_{i+1}\ge n_i+m_i+1$ for all $i\in\N$, and each word  $b_1\cdots b_{m_i-k}$ for $1\le k\le m_i-1$ cannot be followed by other words (see \cite{CLR}). This means that it is extremely difficult to change an admissible word into a word which can be followed by any words.  

In this paper we consider the case when $\a=\textcolor{black}{\k}/\b$, that is, $a=(0, \textcolor{black}{\k, \k, \k,\cdots})$. Then we will show that by replacing at most three characters we can change any admissible word $w$ into a word $\widetilde{w}$ which can be followed by any words (see Remark~\ref{w^coor} and Lemma~\ref{PS2.2}). 
This enables us to give a necessary and sufficient condition on $\b$ such that any equilibrium measure for a function of bounded total oscillations is a weak Gibbs measure in the natural extension of $(\textcolor{black}{\k}/\b, \b)$-shift.

Our proof of the main theorem is based on the method developed  in \cite{PS1}. 
This paper is organized as follows. 

\color{black}
In \S\ref{S.pre}, we provide some definitions and notation. 
The main result is stated in \S\ref{S.main}.
In \S\ref{S.prop}, for each word $w$ we construct a transformation $w\mapsto\widetilde{w}$, which plays a key role in this paper. To investigate its properties, in \S\ref{S.graph}, we introduce a labeled graph that  describes the shift space $\Xab$.
Finally, we prove the main theorem in \S\ref{S.proof}.

\color{black}
\section{\textcolor{black}{Setting and main result}}
\subsection{$(\a, \b)$-shift}\label{S.pre}

Let $0\le\a<1, \b >1$ fixed and $\Tab$ be the transformation defined by \eqref{Tab}. 
For $t\in\R$, write $\lceil t\rceil\coloneq\min\{i\in\Z:i\ge t\}$.
Let $\lambda=\lceil\a+\b\rceil -1$ and $\A=\{0, 1,\cdots, \lambda\}$. We set the partition
$$J_0=\bigg[0, \frac{1-\a}{\b}\bigg), J_k=\bigg[\frac{k-\a}{\b}, \frac{k+1-\a}{\b}\bigg),\cdots, J_\lambda=\bigg[\frac{\lambda-\a}{\b}, 1\bigg)$$
of $[0, 1)$. For $x\in[0, 1)$, the sequence $i_{\a, \b}(x)\in\{0, 1,\cdots, \lambda\}^\N$ is defined by $i_{\a, \b}(x)_n=k$ if and only if $\Tab^{n-1}(x)\in J_k$ for all $n\in\N$. Then we have
\begin{align*}
x&=\frac{i_{\a, \b}(x)_1-\a}{\b}+\frac{i_{\a, \b}(x)_2-\a}{\b^2}+\cdots\\
&=\sum_{n=1}^\infty\frac{i_{\a, \b}(x)_n-\a}{\b^n}.
\end{align*}
The sequence $i_{\a, \b}(x)$ is called an {\it$(\a, \b)$-expansion} of $x$.
Now we define 
$$\Xab\coloneq\text{cl}\{i_{\a, \b}(x)\in \A^\N : x\in\[0, 1)\}$$
where cl$(X)$ is closure of $X$ in the product space $\A^\N$ of discrete sets $\A$. Let $\s:\A^\N\to\A^\N$ be the left shift operator. Then we have $\s(\Xab)\subset\Xab$. We call $(\Xab, \s)$ the {\it(one-sided)$(\a, \b)$-shift}.  

Let $\Lab$ be the language of the shift $\Xab$ and $\Lab_n$ be the set of the words of length $n$.
In this paper we write $\e$ as the empty word, and so $\Lab_0= \{\e\}$.  
A word is written as $w=w_1\cdots w_n$ and the length of a word $w$ is denoted by $|w|$.

\begin{defi}
The natural extension $\Sab$ of $\Xab$ is
$$\Sab=\{x\in \A^\Z:\text{for all}\,\,k\in\Z, (x_k, x_{k+1},\cdots)\in\Xab\}.$$
This is also called the {\it$(\a, \b)$-shift}. \textcolor{black}{The language of $\Sab$ is also $\Lab$. }
\end{defi}


 
For sequences $(x_1, x_2,\cdots)$ and $(y_1, y_2,\cdots)$ in $\A^\N$ the {\it lexicographical order} is defined by $(x_1, x_2,\cdots)\prec(y_1, y_2,\cdots)$ if and only if for the smallest index $i$ with $x_i\ne y_i, x_i<y_i$. 
It is well known that the lexicographical order and the expansion of the end points characterize the $(\a, \b)$-shift $\Xab$:
$$\Xab=\{x\in \A^\N:a\preceq \s^k x\preceq b\,\,\text{for all}\,\, k\in\Z_+\}$$
where $a=i_{\a, \b}(0)$ and $b=\lim\limits_{x\uparrow1}i_{\a, \b}(x)$. In particular, $a=(a_1, a_2,\cdots)$ and $b=(b_1, b_2,\cdots)$ satisfy that $a_1=0$, $b_1=\lambda$, $a\preceq \s^k a\preceq b$ and $a\preceq \s^k b\preceq b$ for all $k\in\Z_+$.

\begin{cor}\cite[Corollary 3.3]{CLR}\label{CLR3.3}
A word $w_1\cdots w_n$ is in $\Lab$ if and only if
$$a_1\cdots a_{n-k+1}\preceq w_k\cdots w_n \preceq b_1\cdots b_{n-k+1}$$
for all $k\in\N$.
\end{cor}



\subsection{\textcolor{black}{Main result}}\label{S.main}
For a word $w$, we set
\begin{align*}
k_1(w)&=\max\{0\}\cup\{1\le k\le |w|:w_{|w|-k+1}\cdots w_{|w|}=a_1\cdots a_k\},\\
k_2(w)&=\max\{0\}\cup\{1\le k\le |w|:w_{|w|-k+1}\cdots w_{|w|}=b_1\cdots b_k\}.
\end{align*}
Then for $w$, 
\begin{align}\label{k}
k_1(w)\neq k_2(w) \quad\text{or}\quad k_1(w)=k_2(w)=0
\end{align}
since $a_1\neq b_1$. 
Let $\Pa\coloneq\{\e, a_1, a_1a_2, a_1a_2a_3,\cdots\}, \Pb\coloneq\{\e, b_1, b_1b_2, b_1b_2b_3,\cdots\}$ and $\P\coloneq\Pa\cup\Pb$.
For any word $w\in\Lab$ there are two ways of decomposition of $w$ into
$$w=v^as^a(w)=v^bs^b(w)$$
where $s^a(w)\in\Pa$ and $s^b(w)\in\Pb$ are the largest suffixes of $w$. 
It follows that $|s^a(w)|=k_1(w)$ and $|s^b(w)|=k_2(w)$. 
Furthermore there is a unique decomposition of $w$ into
\begin{eqnarray}
w=vs(w)\quad\text{where}\,\,s(w)=
   \begin{cases}
      s^a(w) & \,\,\text{if}\,\,\,k_1(w)>k_2(w), \\
      s^b(w) & \,\,\text{if}\,\,\,k_1(w)<k_2(w), \\
      \e & \,\,\text{if}\,\,\,k_1(w)=k_2(w)=0.
   \end{cases}\label{decom}
\end{eqnarray}
(Recall that \eqref{k} holds.)
For $u\in\Pb$ we set 
\begin{align}\label{zab}
z^{\a, \b}(u)\coloneq
\begin{cases}
   0 \quad\text{if}\,\,\,u=\e, \\
   \max\{k\in\N:ua_{k_1(u)+1}\cdots a_{k_1(u)+k}\in\Pb\}\cup\{0\} \quad\text{if}\,\,\,u\neq\e
 \end{cases}
\end{align}
and
$$\bar{z}^{\a, \b}(n)\coloneq\max\{z^{\a, \b}(u):u\,\,\text{prefix of $b$}, |u|\le n\}.$$
By setting
\begin{eqnarray*}
z^{\a, \b}(w)\coloneq
  \begin{cases}
  0 & \quad\text{if}\,\,\, k_1(w)\ge k_2(w),\\
  z^{\a, \b}(s^b(w)) & \quad\text{if}\,\,\,k_1(w)<k_2(w),
  \end{cases}
\end{eqnarray*}
the definition \eqref{zab} is extended to any word $w\in\Lab$.

\begin{defi}
An invariant probability measure $\nu$ is called a {\it weak Gibbs measure for a continuous function} $\psi$, if for any $\delta>0$ there exists $N_\delta\in\N$ such that for $m\ge N_\delta$,
\begin{align}\label{weakGibbs}
\sup_{x\in\Sab}\left|\frac{1}{m}\ln\nu([x_1\cdots x_{m}])-\frac{1}{m}\sum_{l=1}^{m}\psi(\s^l(x))\right|\le\delta,
\end{align}
where $[x_1\cdots x_{m}]=\{y\in\Sab:y_1\cdots y_{m}=x_1\cdots x_{m}\}$.
\end{defi}

\begin{remark}
 If we set $\psi=\varphi-p(\varphi)$ (the pressure $p(\varphi)$ will be defined later), then \eqref{weakGibbs} is equivalent to the condition that for $x\in\Sab$
$$e^{-\delta m}\le\frac{\nu([x_1\cdots x_{m}])}{\exp(-mp(\varphi)+\sum\nolimits_{l=1}^{m}\varphi(\s^l x))}\le e^{\delta m}.$$
Pfister and Sullivan \cite{PS2} proved that if $\nu$ is a weak Gibbs for $\psi$, then $p(\varphi)=0$.
\end{remark}

Our main results are following.
\begin{thm}\label{main}
Let \textcolor{black}{$\b>2$, $\k\in\N_0$ with $\k<\lceil\k/\b+\b\rceil -2$}, $\a=\textcolor{black}{\k}/\b$ and $\varphi\in C(\Sab)$ with bounded total oscillations (see \S\ref{bto} for a definition).
\begin{enumerate}
     \item If $\nu$ is an equilibrium measure for $\varphi$ and if
     $$\lim_{n\to\infty}\frac{\bar{z}^{\a, \b}(n)}{n}=0,$$
     then $\nu$ is a weak Gibbs measure for $\psi=\varphi-p(\varphi)$.
     \vspace{0.2cm}
     \item If $\nu$ is an equilibrium measure for $\varphi$ and if
     \begin{align}\label{nweakGibbs}
     \limsup_{n\to\infty}\frac{\bar{z}^{\a, \b}(n)}{n}>0,
     \end{align}
     then $\nu$ is not a weak Gibbs measure for $\psi=\varphi-p(\varphi)$.
\end{enumerate}
\end{thm}

\color{black}
In the case when $\kappa=0$, $(\kappa/\b, \b)$-shifts are $\b$-shifts. 
Thus for $\b>2$ Theorem~2.12 in \cite{PS1} can be consideblack as a corollary of Theorem~\ref{main}.
Indeed, since $a=i_{0, \b}(0)=(0,0,\cdots)$, we can rewrite the definition \eqref{zab} as
\begin{align*}
z^\b(u)\coloneq z^{0, \b}(u)=
\begin{cases}
   0 \quad\text{if}\,\,\,u=\e, \\
   \max\{k\in\N:u0^k\in\Pb\}\cup\{0\} \quad\text{if}\,\,\,u\neq\e,
 \end{cases}
\end{align*}
for $u\in\Pb$, where $0^k=\underbrace{0 \cdots \cdots 0}_{k}$.
We also write $\Sigma^\b\coloneq\Sigma^{0, \b}$ and $\bar{z}^\b(n)\coloneq\bar{z}^{0, \b}(n)$ for $n\in\N$. Then we obtain the following corollary for $\b$-shifts:

\begin{cor}[{\cite[Theorem 2.12]{PS1}}]
Let $\b>2$ and $\varphi\in C(\Sigma^\b)$ have bounded total oscillations. If $\nu$ is an equilibrium measure for $\varphi$, then $\nu$ is a weak Gibbs measure for $\psi=\varphi-p(\varphi)$ if and only if $\lim_{n\to\infty}\bar{z}^\b(n)/n=0.$
\end{cor}

\color{black}
\begin{ex}
Carapezza et al. established some criterion for the specification property of $(\a, \b)$-shifts (\cite{CLR}). For $(1/\b, \b)$-shifts their criterion implies that $(\Sigma_{1/\b, \b}, \s)$ satisfies the specification property if and only if there exists $M>0$ such that $\bar{z}^{\a, \b}(n)<M$ for all $n\in\N$. Thus by Theorem~\ref{main}  if $(\Sigma_{1/\b, \b}, \s)$ has this property, then any equilibrium measure for $\varphi$ is a weak Gibbs measure for $\psi=\varphi -p(\varphi)$. 
\end{ex}

\begin{ex}
By using arguments as in \cite{OS} we can choose a parameter $\b$ satisfying the condition \eqref{nweakGibbs}.
Set $n_k=2k-1+2^k$ for $k\ge 0$. Let $\b_0>3$ be the larger root of $\b+1/\b =4$ and $I_0=(\b_0, 4)$. We define
\begin{align*}
b_j=
\begin{cases}
   4\quad\text{if}\;\;j=1,\\
   0\quad\text{if}\;\;j=n_k+2, k\ge 0,\\
   2\quad\text{if}\;\;j=n_k+1, k\ge 1,\\
   1\quad\text{otherwise}.
\end{cases}
\end{align*}
That is, 
\begin{align*}
b&=(b_1, b_2, \cdots)\\
&=(4, 0, 1, 2, 0, \underbrace{1, 1}_{2^1}, 2, 0, \underbrace{1, 1, 1, 1}_{2^2}, 2, 0, \underbrace{1, \cdots, 1}_{2^3}, 2, \cdots, 0, \underbrace{1, \cdots, 1}_{2^k}, 2, \cdots).
\end{align*}
Fix $n\in\N$. Set $\mathcal{N}_n=\{\b>3:T_{1/\b, \b}^{k-1}(1)\in J_{b_{k}}\;\text{for all}\; 1\le k\le n\}$ (see Figure~\ref{trans}).
For $x\in\mathcal{N}_n$, we define $P_1(x)=x+1/x -b_1$ and $P_k(x)=xP_{k-1}(x)+1/x -b_k$ for $1\le k\le n$. 
Then we have $T_{1/\b, \b}^{k-1}(1)=P_k(\b)$ for $1\le k\le n$. 
Since $P_1'(x)=1-1/x^2$ and $P_k'(x)=P_{k-1}(x)+xP_{k-1}'(x)-1/x^2$ for $2\le k\le n$, we get
$$P_n'(x)=x^{n-1}\left((1-\frac{1}{x^2})x+\sum_{j=1}^{n-1}\frac{P_j(x)-1/x^2}{x^j}\right)\ge\frac{5}{2}x^{n-1}$$
and $P_n(x)$ is strictly increasing for $x>3$.

We prove $\bigcap_{n\ge 1}\mathcal{N}_n\neq\emptyset$. 
Set
$$J_0(\b)=\left[0, \frac{1-1/\b}{\b}\right), J_k(\b)=\left[\frac{k-1/\b}{\b}, \frac{k+1-1/\b}{\b}\right)\quad\text{for}\;k=1, 2, 3.$$
For $j\ge 1$ we define 
$$I_j=\left\{\b\in I_{j-1}:\begin{array}{l}
P_{n_{j-1}+2}(\b)\in J_0(\b), P_{n_{j-1}+i}(\b)\in J_1(\b)\;\text{for}\;2< i\le n_j-n_{j-1},\\
P_{n_j+1}(\b)\in J_2(\b)\end{array}\right\},$$
where $I_0=(3, 4)$.
By induction we can show that if we set $I_j=[\gamma, \eta)$, then we have $P_{n_j+1}(I_j)=[\frac{2-1/\gamma}{\gamma}, \frac{3-1/\eta}{\eta})\neq\emptyset$ because $P_n$ is expanding and  satisfies $P_{n_j+2}(\gamma)=0$, $P_{n_i+3}(\gamma)=1/\gamma$ for $i\ge 3$. 
Since $I_j\subset \mathcal{N}_{n_{j}+1}$, we have $\bigcap_{n\ge 1}\mathcal{N}_n\supset\bigcap_{j\ge 1}I_j\neq \emptyset$.
\end{ex}

\color{black}

\begin{figure}[htbp]
\vspace*{-1cm}
\captionsetup{labelfont={color=black}, textfont={color=black}}
\centering
\includegraphics[width=80mm]{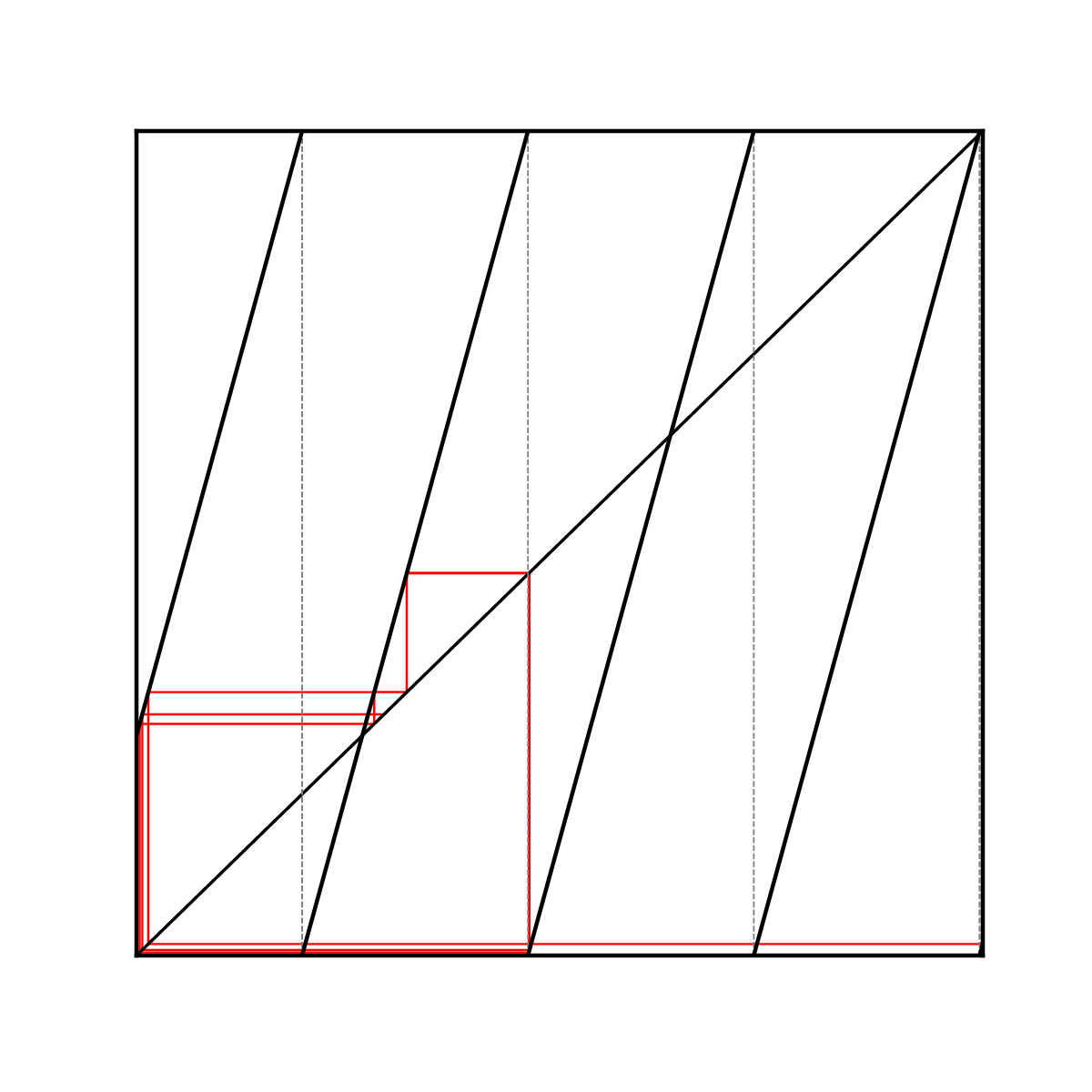}
\vspace{-0.5cm}
\caption{Graph of $T_{1/\b, \b}$ for $\b=3.747135\in\mathcal{N}_9$.}\label{trans}
\end{figure}

\section{\textcolor{black}{Mappings to words which can be freely concatenated}}\label{S.prop}
In this chapter, we introduce a transformation $w\mapsto\widetilde{w}$ of $\Lab$ such that $\widetilde{w}$ can be followed by all words in $\Lab$. 
\subsection{\textcolor{black}{Graphical description}}\label{S.graph}
In this section we consider a labeled graph $\Gab$, which can describe the $(\a, \b)$-shifts $\Xab$. This construction follows \cite{CLR}. 
\color{black}
The {\it follower set} $F(w)$ of the word $w\in\Lab$ is defined by
$$F(w)=\s^{|w|}([w])=\{x\in\A^\N:wx\in\Xab\}.$$
\color{black}
The following fact is an easy consequence of Corollary~\ref{CLR3.3}.
\begin{lem}
For $w\in\Lab$,
$$F(w)=\{x\in\Xab:\s^{k_1(w)}(a)\preceq x\preceq \s^{k_2(w)}(b)\}.$$
\end{lem}

It is known that the $(\a, \b)$-shifts $\Xab$ can be described by a graph whose vertices are given by follower sets 
(see e.g., \cite{Ho2}). In this paper we use the labeled graph $\Gab=(\mathcal{V}^{\a, \b}, \mathcal{E}^{\a, \b})$ introduced by Carapezza, L\'{o}pez and Robertson (\cite{CLR}) as follows : we define the set of vertices by 
$$\mathcal{V}^{\a, \b}=\{[k_1(w), k_2(w)]:w\in\Lab\}\subset\Z_+^2.$$
The set $\mathcal{E}^{\a, \b}$ of edges consists of four types of edges defined in the following way.
\begin{enumerate}
     \item[\Ea\!\!.] 
     If $[j, k]\in\mathcal{V}^{\a, \b}$ and $a_{j+1}=b_{k+1}$ then there is a unique edge from $[j, k]$ to $[j+1, k+1]$ labeled by $a_{j+1}=b_{k+1}$.
     \vspace{0.2cm}
     \item[\Eb\!\!.] 
     If $[j, k]\in\mathcal{V}^{\a, \b}$ and $a_{j+1}<b_{k+1}$ then there is a unique edge from $[j, k]$ to $[j+1, 0]$ labeled by $a_{j+1}$.
     \vspace{0.2cm}
     \item[\Ec\!\!.] 
     If $[j, k]\in\mathcal{V}^{\a, \b}$ and $a_{j+1}<b_{k+1}$ then there is a unique edge from $[j, k]$ to $[0, k+1]$ labeled by $b_{k+1}$.
     \vspace{0.2cm}
     \item[\Ed\!\!.] 
     If $[j, k]\in\mathcal{V}^{\a, \b}$ and $a_{j+1}<c<b_{k+1}$ then there is a unique edge from $[j, k]$ to $[0, 0]$ labeled by $c$.
\end{enumerate}

The following lemmas have been shown by Carapezza, L\'{o}pez and Robertson (see \cite{CLR} for more details).
\begin{lem}[{\cite[Lemma 3.9]{CLR}}] 
Let $\gamma$ be a path in $\Gab$ beginning at $[0,0]$ and ending at $[i, j]$. If $w$ is the associated word over $\A$
then $i=k_1(w)$ and $j=k_2(w)$.
\end{lem}

\begin{lem}[{\cite[Lemma 3.10]{CLR}}]\label{CLR3.10}
A word $w$ over $\A$
 is in $\Lab$ if and only if 
 a path beginning at $[0, 0]$ with label $w$ is in $\Gab$.
\end{lem}

By Lemma \ref{CLR3.10}, every word $w=w_1\cdots w_n\in\Lab$ corresponds to a path $\pth(w)$ in $\Gab$ starting at $[0, 0]$ as follows:
\begin{align*}
\pth(w):&[0, 0]\xrightarrow{w_1}[k_1(w_1), k_2(w_1)]\xrightarrow{w_2}[k_1(w_1w_2), k_2(w_1w_2)]\\
&\hspace{1cm}\xrightarrow{w_3}\cdots\xrightarrow{w_n}[k_1(w), k_2(w)].
\end{align*}
Write $\vtx(w)=[k_1(w), k_2(w)]$. 
\textcolor{black}
{
\begin{remark}
In the case when $\a=0$ we will use the following modified graph $\mathcal{G}^\b=(\mathcal{V}^{\b}, \mathcal{E}^{\b})$ instead of $\Gab$: the set of verticies $\mathcal{V}^{\b}$ is defined by 
$$\mathcal{V}^{\b}=\{[k_1(w), k_2(w)]:w\in\mathcal{L}^{0, \b}, k_2(w)>0\}\cup\{[0, 0]\}\subset\Z_+^2$$ 
and $\mathcal{E}^\b$ is the set of all edges satisfying \Ea\!\!, \Ec\!\!,  \Ed and 
\begin{enumerate}
\item[\Eb\!\!'.] If $[j, k]\in\mathcal{V}^{\a, \b}$ and $a_{j+1}<b_{k+1}$ then there is a unique edge from $[j, k]$ to $[0, 0]$ labeled by $a_{j+1}$.
\end{enumerate}
For $w\in\mathcal{L}^{0, \b}$ we set 
\begin{align*}
k_1'(w)=
\begin{cases}
k_1(w)&k_2(w)>0,\\
0&k_2(w)=0
\end{cases}
\end{align*}
and $\vtx'(w)=[k_1'(w), k_2(w)]$. Then we have $\mathcal{V}^\b=\{\vtx'(w):w\in\mathcal{L}^{0, \b}\}$ and every word $w=w_1\cdots w_n\in\mathcal{L}^{0, \b}$ corresponds to a path 
\begin{align*}
\pth'(w):&[0, 0]\xrightarrow{w_1}[k_1'(w_1), k_2(w_1)]\xrightarrow{w_2}[k_1'(w_1w_2), k_2(w_1w_2)]\\
&\hspace{1cm}\xrightarrow{w_3}\cdots\xrightarrow{w_n}[k_1'(w), k_2(w)].
\end{align*}
Clearly, the map $\pth(w)\mapsto\pth'(w)$ is bijective. By using $\mathcal{G}^\b$ we can consider the self-loop $[0, 0]\xrightarrow{0}[0, 0]$. 
\end{remark}
}

\def\Vab{\mathcal{V}^{\a, \b}}

\subsection{\textcolor{black}{The transformation $w\mapsto\widetilde{w}$ of $\Lab$}}
From now on we assume that \textcolor{black}{$\b>2$ and $\a=\k/\b$, where $\k\in\N_0$ satisfies $\k<\lceil\k/\b+\b\rceil -2$}. One can easily check that 
\begin{align*}
a &=(a_1, a_2, \cdots)=(0, \k, \k, \k\cdots),\\
b_1 &=\lambda=\lceil \k/\b+\b\rceil-1\ge\textcolor{black}{\kappa+2}.
\end{align*}

\begin{figure}[htbp]
\centering
\hspace*{-1.8cm}
\includegraphics[clip, width=140mm]{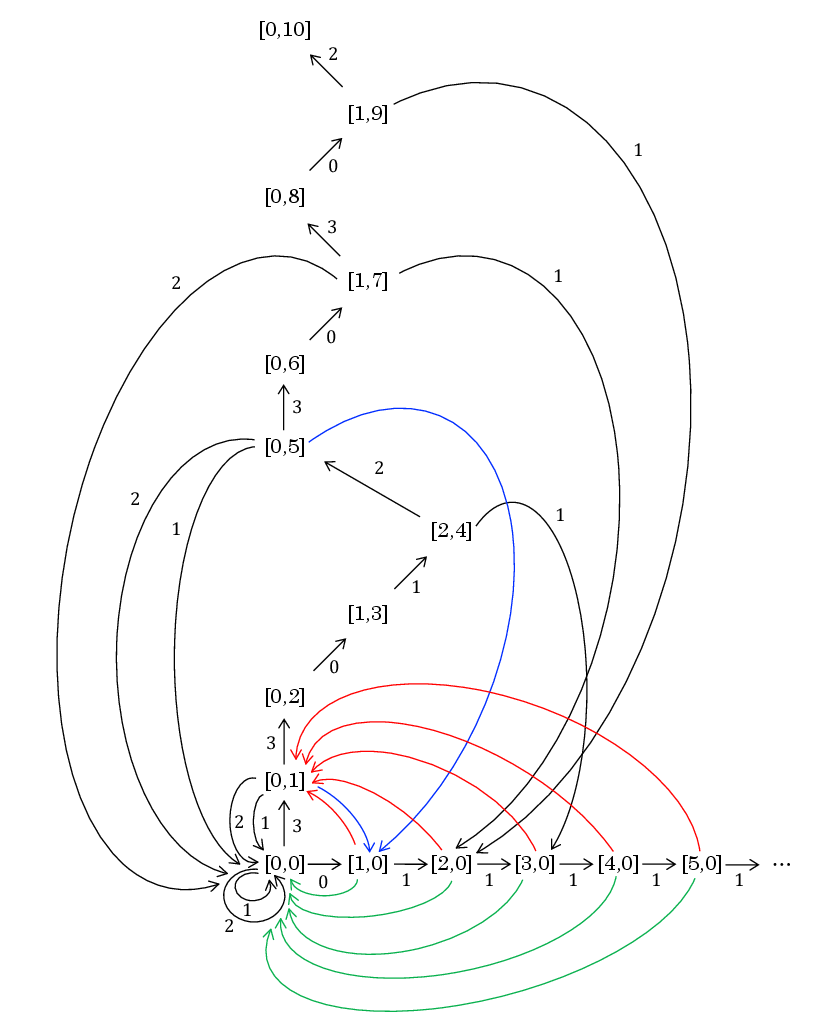}
\caption{Part of the graph $\Gab$ when $\beta=7/2$, $\a=2/7$ and $b=(3, 3, 0, 1, 2, 3, 0, 3, 0, 2,\cdots)$ with all edges terminating at $[0, 1]$ labeled ``$3$'', all edges terminating at $[1, 0]$ labeled ``$0$'' and all edges from $[j, 0]$ to $[0, 0]$ labeled ``$2$''.}
\end{figure}

The following lemmas are easy consequences of the definition of $\Gab$.

\begin{lem}\label{grapha}
The edge $[0, 0]\to[1, 0]$ in $\Gab$ is labeled by $0$ and each edge $[k, 0]\to[k+1, 0]$ ($k\ge1$) in $\Gab$ is labeled by $\k$. For $k\ge1$ and $\k< i\le\lambda-1$ there are an edge from $[k, 0]$ to $[0, 0]$ labeled by $i$ and an edge from $[k, 0]$ to $[0, 1]$ labeled by $\lambda$.
\end{lem}

\begin{lem}\label{graphb}
The edge $[0, 0]\to[0, 1]$ in $\Gab$ is labeled by $\lambda$. If  $[j, k]$ is in $\Vab$ for $j, k\ge1$, then we have $j<k$, the edge $[0, k-j]\to[1, k-j+1]$ is labeled by $0$ and for $1\le i\le j-1$ the edge $[i, k-j+i]\to[i+1, k-j+i+1]$ is labeled by $\k$.
\end{lem}

For each prefix $u\in\P$ we define a word $\widehat{u}$ as follows:
 For $u=a_1\cdots a_l\in\Pa$, $\widehat{u}=\widehat{a}_1\cdots\widehat{a}_l$
is defined by
\begin{eqnarray*}
\widehat{a}_i
\coloneq 
  \begin{cases}
     a_i & \quad\text{if}\,\,\,i=1,\cdots, l-1, \\
     a_l+1 & \quad\text{if}\,\,\,i=l.
  \end{cases}
  \end{eqnarray*}
For $u=b_1\cdots b_l\in\Pb$ with $b_l<\textcolor{black}{\k+2}$, $\widehat{u}\coloneq\widehat{b}_1\cdots\widehat{b}_l$ is defined by
\begin{eqnarray*}
\widehat{b}_i
\coloneq 
  \begin{cases}
     b_i & \quad\text{if}\,\,\,i\neq l-k_1(u^-)-1, l-k_1(u^-), \\
     b_i-1 & \quad\text{if}\,\,\,i=l-k_1(u^-)-1,\\
     \lambda-1 & \quad\text{if}\,\,\,i=l-k_1(u^-)
  \end{cases}
  \end{eqnarray*}
where $u^-=b_1\cdots b_{l-1}$.
For $u=b_1\cdots b_l\in\Pb$ with $b_l\ge\textcolor{black}{\k+2}, \widehat{u}\coloneq\widehat{b}_1\cdots\widehat{b}_l$ is defined by
\begin{eqnarray*}
\widehat{b}_i
& =
  \begin{cases}
     b_i & \quad\text{if}\,\,\,i=1,\cdots, l-1, \\
     b_l-1 & \quad\text{if}\,\,\,i=l.
  \end{cases}
\end{eqnarray*}
By convention we set $\widehat{\e}=\e$.

\begin{remark}\label{u^coor}
If $u=\e$ then $u=\widehat{u}$.
If $u\in\Pa\backslash\{\e\}$ then $u$ and $\widehat{u}$ differ at the last coordinate.
If $u\in\Pb\backslash\{\e\}$ then $u$ and $\widehat{u}$ differ at most at two coordinates.
\end{remark}

\begin{lem}\label{u^}
Let $u\in\P$. Then 
$\widehat{u}\in\Lab$ and $\vtx(\widehat{u})=[0, 0]$.
\end{lem}

\begin{proof}
If $u=\e$, this is clear since $\widehat{u}=\e$. 
Let $u=a_1\cdots a_l\in\Pa\backslash\{\e\}$. 
Then $\widehat{u}=\widehat{a}_1\cdots\widehat{a}_l$ where $\widehat{a}_1\cdots\widehat{a}_{l-1}=a_1\cdots a_{l-1}$ and 
$$a_l<\widehat{a}_l=a_l+1\le \textcolor{black}{\k+1}\le \lambda-1.$$
Thus $\pth(\widehat{a}_1\cdots\widehat{a}_{l-1})=\pth(a_1\cdots a_{l-1})$ is in $\Gab$, and by Lemma~\ref{grapha} the edge $[l-1, 0]\to[0, 0]$ labeled by $\widehat{a}_l=a_l+1$ is in $\Eab$. 
By Lemma~\ref{CLR3.10} the statement of the lemma follows. 

Let $u=b_1\cdots b_l\in\Pb\backslash\{\e\}$. 
If $b_l\ge\textcolor{black}{\k+2}$, then $\widehat{b}_{l}=b_l-1$ and $\widehat{b}_{i}=b_i$ for $i=1,\cdots,l-1$. 
By Lemma~\ref{CLR3.10}, $\pth(\widehat{b}_1\cdots\widehat{b}_{l-1})=\pth(b_1\cdots b_{l-1})$ is in $\Gab$. 
We can write $[j, l-1]=\vtx(b_1\cdots b_{l-1})$ for $j=k_1(b_1\cdots b_{l-1})$. 
Since $a_{j+1}(=0\,\text{or}\,\k)<b_l-1<b_l$, it follows from \Ed that the edge $[j, l-1]\to[0, 0]$ labeled by $\widehat{b}_{l}=b_l-1$ is in $\Eab$. 
Therefore we get $\pth(\widehat{u})$ is in $\Gab$ and $\vtx(\widehat{u})=[0, 0]$.

If $b_l<\textcolor{black}{\k+2}$, then $\widehat{b}_{i}=b_i$ for $i\neq l-k_1(u^-)-1, l-k_1(u^-)$ where $u^-=b_1\cdots b_{l-1}$ and $\widehat{b}_{l-k_1(u^-)-1}=b_{l-k_1(u^-)-1}-1$ and $\widehat{b}_{l-k_1(u^-)}=\lambda-1$.   
Then 
$$\pth(\widehat{b}_1\cdots\widehat{b}_{l-k_1(u^-)-2})=\pth(b_1\cdots b_{l-k_1(u^-)-2})$$
is in $\Gab$. 
We can write 
$$[j, l-k_1(u^-)-2]=\vtx(b_1\cdots b_{l-k_1(u^-)-2})$$
for $j=k_1(b_1\cdots b_{l-k_1(u^-)-2})$.

(Case 1) : When $b_{l-k_1(u^-)-1}\ge\textcolor{black}{\k+2}$, we have $a_{j+1}(=0\,\text{or}\,\k)<b_{l-k_1(u^-)-1}-1<b_{l-k_1(u^-)-1}$. Thus there is an edge from $[j, l-k_1(u^-)-2]$ to $[0, 0]$ labeled by $b_{l-k_1(u^-)-1}-1$ and $\pth(\widehat{b}_{1}\cdots\widehat{b}_{l-k_1(u^-)-1})$ is in $\Gab$. 
Since $\widehat{b}_{l-k_1(u^-)}=\lambda-1$, $\widehat{b}_i=\k\, (l-k_1(u^-)+1\le i\le l-1)$ and 
\begin{align*}
\widehat{b}_l=
\left\{\begin{array}{ll}\k \,\,\text{or}\,\, \textcolor{black}{\k+1} & (k_1(u^-)>0),\\
\lambda-1 & (k_1(u^-)=0),\end{array}\right. 
\end{align*}
all edges in $\pth(\widehat{b}_{l-k_1(u^-)}\cdots\widehat{b}_{l})$ are self-loops on $[0, 0]$. 

(Case 2) : The case of $b_{l-k_1(u^-)-1}=\textcolor{black}{\k+1}$ and $j=0$ is similar to Case 1.

(Case 3) : When $b_{l-k_1(u^-)-1}=\textcolor{black}{\k+1}$ and $j>0$, $\widehat{b}_{l-k_1(u^-)-1}=\k$ labels the edge $[j, l-k_1(u^-)-2]\to[j+1, 0]$ and $\widehat{b}_{l-k_1(u^-)}=\lambda-1$ labels the edge $[j+1, 0]\to[0, 0]$. 
By the same argument as in Case 1, 
all edges in $\pth(\widehat{b}_{l-k_1(u^-)+1}\cdots\widehat{b}_{l})$ are self-loops on $[0, 0]$. Thus we have $\pth(\widehat{u})$ is in $\Gab$ and $\vtx(\widehat{u})=[0, 0]$.

(Case 4) : When $b_{l-k_1(u^-)-1}=\textcolor{black}{\k>0}$, we have $j=0$. \textcolor{black}{If $\k=1$}, then $\widehat{b}_{l-k_1(u^-)-1}=0$ labels the edge $[0, l-k_1(u^-)-2]\to[1, 0]$ and $\widehat{b}_{l-k_1(u^-)}=\lambda-1$ labels the edge $[1, 0]\to[0, 0]$.
\textcolor{black}{If $\k\ge2$, then $\widehat{b}_{l-k_1(u^-)-1}=\k-1$ and $0<\widehat{b}_{l-k_1(u^-)-1}<\lambda$. Thus $\widehat{b}_{l-k_1(u^-)-1}$ labels the edge $[0, l-k_1(u^-)-2]\to[0, 0]$ and $\widehat{b}_{l-k_1(u^-)}=\lambda-1$ labels the edge $[0, 0]\to[0, 0]$.}
By the same argument as in Case 1 all edges in $\pth(\widehat{b}_{l-k_1(u^-)+1}\cdots\widehat{b}_{l})$ are self-loops on $[0, 0]$. 
\textcolor{black}{Therefore} $\pth(\widehat{u})$ is in $\Gab$ and $\vtx(\widehat{u})=[0, 0]$.
\end{proof}

We extend the transformation $\Lab\ni w\mapsto\widehat{w}$ to any word $w\in\Lab$ by
\begin{eqnarray*}
\widehat{w}
\coloneq 
  \begin{cases}
     w & \quad\text{if}\,\,\,s(w)=\e, \\
     v\widehat{u} & \quad\text{if}\,\,\,s(w)=u\neq\e.
  \end{cases}
\end{eqnarray*}
We define the transformation $\Lab\ni w\mapsto\widetilde{w}$ by
\begin{eqnarray}
\widetilde{w}
\coloneq\widehat{(\widehat{w})}\coloneq
  \begin{cases}
     \widehat{w} & \quad\text{if}\,\,\,s(\widehat{w})=\e, \\
     v'\widehat{u'} & \quad\text{if}\,\,\,s(\widehat{w})=u'\neq\e
  \end{cases}\label{hat}
\end{eqnarray}
where $\widehat{w}=v's(\widehat{w})$. 

\begin{lem}\label{00}
For $w\in\Lab$, $s(w)=\e$ if and only if $\vtx(w)=[0, 0]$.
\end{lem}

\begin{proof}
Since $\vtx(w)=[k_1(w), k_2(w)]$ and $|s(w)|=\max\{k_1(w), k_2(w)\}$, $\vtx(w)=[0, 0]$ if and only if $|s(w)|=0$, and hence if and only if $s(w)=\e$.
\end{proof}

The following lemma can be proved in the same way as Lemma 2.1 in \cite{PS1}.
\begin{lem}[{\cite[Lemma 2.1]{PS1}}]\label{PS2.1}
\mbox{}
\begin{enumerate}
     \item Let $c=c_1\cdots c_k$ and $d=d_1\cdots d_l$ be two prefixes of $a$. If $cd\in\Lab$, then $cd$ is a prefix of $a$.
     \vspace{0.2cm}
     \item Let $c=c_1\cdots c_k$ and $d=d_1\cdots d_l$ be two prefixes of $b$. If $cd\in\Lab$, then $cd$ is a prefix of $b$.
\end{enumerate}
\end{lem}

\begin{lem}\label{w^}
Let $w\in\Lab$.
\begin{enumerate}
 \item If $s(w)\in\Pb$ then $\widehat{w}\in\Lab$ and $s(\widehat{w})=\e$.
 \item If $s(w)\in\Pa$ then $\widehat{w}\in\Lab$ and $s(\widehat{w})\in\Pb$.
 \item We have $\widetilde{w}\in\Lab$ and $s(\widetilde{w})=\e$.
\end{enumerate}
\end{lem}

\begin{proof}
Let $w=vs(w)\in\Lab$. If $s(w)=\e$, then it is clear that $s(\widehat{w})=\e$. 
Now assume $s(w)\neq\e$. If $\vtx(v)=[0, 0]$, it follows from Lemma~\ref{u^} that $\pth(\widehat{w})$ is in $\Gab$ and $\vtx(\widehat{u})=\vtx(\widehat{w})=[0, 0]$, which means $s(\widehat{w})=\e$.

(i) We may assume that $w=vu, s(w)=u\in\Pb\backslash\{\e\}$ and $\vtx(v)\neq[0, 0]$. 
We write 
$$s(w)=u=b_1\cdots b_l, l=k_2(w)>0 \quad\text{and}\quad u^-=b_1\cdots b_{l-1}.$$
Since $\vtx(v)\neq[0, 0]$, Lemma~\ref{PS2.1}(ii) implies that $\vtx(v)=[j, 0]$ for some $j>0$.
By Lemma~\ref{grapha}, $b_1$ labels the edge $[j, 0]\to[0, 1]$. We divide the proof into the following four cases. In these cases we will show $\vtx(\widehat{w})=[0,0]$, which implies $s(\widehat{w})=\e$. 

Let $p_1=\Zab(b_1)$. If $p_1>0$, by \eqref{zab} we have 
$$b_2b_3\cdots b_{p_1+1}=\textcolor{black}{0\k\cdots \k}, b_{p_1+2}\ge\textcolor{black}{\k+1}$$
and 
\begin{align*}
[j, 0]\xrightarrow{b_1}[0, 1]\xrightarrow{b_2}[1, 2]\xrightarrow{b_3}\cdots\xrightarrow{b_{p_1+1}}[p_1, p_1+1]\xrightarrow{b_{p_1+2}}[0, p_1+2].
\end{align*}
If $p_1=0$ we have $b_2\ge1$ and $[j, 0]\xrightarrow{b_1}[0, 1]\xrightarrow{b_2}[0, 2]$.


(Case i-1): When $l\ge p_1+3$ one can easily check that $l-k_1(u^-)-1\ge2$, which means that $\widehat{b}_1=b_1$. 
\textcolor{black}{Hence} this case can be shown by Lemma~\ref{u^}.


(Case i-2): When $l=1$ we have $s(w)=u=b_1$ and $\vtx(v)=[j, 0]$ for some $j>0$.
Then there exists only one edge $\vtx(v)=[j, 0]\to[0,0]$ labeled by $\widehat{b}_1=b_1-1=\lambda-1$. Hence $\pth(\widehat{w})=\pth(v\widehat{b}_1)$ is in $\Gab$ and $\vtx(\widehat{w})=\vtx(v\widehat{b}_1)=[0, 0]$.

(Case i-3): The case when $l=p_1+2$ and $b_l\ge\textcolor{black}{\k+2}$ can be proved easily by Lemma~\ref{u^} because of $\widehat{b}_1=b_1$.

(Case i-4): Suppose that $l=p_1+2$ and $b_l=\textcolor{black}{1, \cdots, \k+1}$ or that $2\le  l\le p_1+1$ and $p_1>0$.
In the case, since $b_2\cdots b_{l-1}=a_1\cdots a_{l-2}$, we have $k_1(u^-)=l-2$ and $b_{l-k_1(u^-)-1}=b_1$. 
Therefore we change $b_1$ and $b_2$ into $\lambda-1$ by the definition of $\widehat{u}$. 
We have $\vtx(v)=[j, 0]$ for some $j>0$ and there exists the edge $[j, 0]\to[0, 0]$ labeled by $\lambda-1$ and all edges in $\pth(\widehat{b_2}\cdots\widehat{b}_{l})=\pth(\widehat{b}_2b_3\cdots b_l)$ are self-loops on $[0, 0]$ since $0<\widehat{b}_2=\lambda-1<\lambda$ and $0<b_i<\lambda$ for $i=3,\cdots, l$. 
Therefore $\pth(\widehat{w})=\pth(v\widehat{u})$ is in $\Gab$ and $\vtx(\widehat{w})=[0, 0]$.


(ii) Let $w=vs(w)$, $s(w)\in\Pa\backslash\{\e\}$ and $\vtx(v)\neq[0, 0]$. 
We write 
$$s(w)=u=a_1\cdots a_l, l=k_1(w)>0\quad\text{and}\quad u^-=a_1\cdots a_{l-1}.$$
Since $\widehat{a}_i=a_i$ for $i=1,\cdots, l-1$, it follows that $\pth(vu^-)$ is in $\Gab$.

(Case ii-1): If $k_2(vu^-)=0$, then 
$$\vtx(vu^-)=[k_1(vu^-), k_2(vu^-)]=[l-1, 0].$$ 
Since $a_l<a_l+1=\widehat{a}_l<\lambda=b_1$, $\widehat{a}_l=a_l+1$ labels the edge $[l-1, 0]\to[0, 0]$.
Therefore $\pth(\widehat{w})=\pth(v\widehat{u})$ is in $\Gab$ and $\vtx(\widehat{w})=\vtx(\widehat{u})=[0, 0]$, 
\textcolor{black}{and thus} $\widehat{w}\in\Lab$ and $\vtx(\widehat{w})=[0, 0]$.

(Case ii-2): If $k\coloneq k_2(vu^-)>0$, then $\vtx(vu^-)=[l-1, k]$.
Since $k>0$ and $k_2(w)=0$, we have $b_{k+1}\ge a_l+1=\widehat{a}_l$.
If $b_{k+1}>\widehat{a}_l$, then $\widehat{a}_l$ labels the edge $[l-1, k_2(vu^-)]\to[0, 0]$ by \Ed because $a_l<\widehat{a}_l<b_{k+1}$.
Hence $\pth(\widehat{w})$ is in $\Gab$ and $\vtx(\widehat{w})=[0, 0]$.

If $b_{k+1}=\widehat{a}_l$, then there is no edge $[l-1, k_2(vu^-)]\to[0, 0]$.
By the rule \Ec the edge $[l-1, k]\to[0, k+1]$ is labeled by $b_{k+1}=\widehat{a}_l$. 
Thus $\pth(\widehat{w})=\pth(v\widehat{u})$ is in $\Gab$ and $\widehat{w}\in\Lab$. 
However, in this case 
$$\vtx(\widehat{w})=[0, k+1]\neq[0, 0].$$ 
Let $\widehat{w}=v's(\widehat{w})$. Since $\vtx(\widehat{w})=[0, k+1]$, we have $k_1(\widehat{w})=0$, $k_2(\widehat{w})=k+1>0$ and $s(\widehat{w})\in\Pb\backslash\{\e\}$.

(iii) Let $w\in\Lab$. If $s(w)=\e$, then it is clear that $s(\widetilde{w})=\e$ because $\widetilde{w}=w$. 

If $s(w)\in\Pb\backslash\{\e\}$, then $s(\widehat{w})=\e$ by (i). Therefore $\widetilde{w}=\widehat{w}\in\Lab$ and $s(\widetilde{w})=s(\widehat{w})=\e$.

Let $s(w)\in\Pa\backslash\{\e\}$. By (ii) we have $s(\widehat{w})=\e$ or $s(\widehat{w})\in\Pb\backslash\{\e\}$. If $s(\widehat{w})=\e$ then $\widetilde{w}=\widehat{w}$ and we get $s(\widetilde{w})=\e$. When $s(\widehat{w})\in\Pb\backslash\{\e\}$, this case has already proved in (i).
\textcolor{black}{Hence} $\widetilde{w}=\widehat{(\widehat{w})}\in\Lab$ and $\vtx(\widetilde{w})=[0, 0]$, that is, $s(\widetilde{w})=\e$.
\end{proof}

\def\La{\mathcal{L}^{a}}
\def\Lb{\mathcal{L}^{b}}
\def\Le{\mathcal{L}^{\e}}

Let 
$$\Lab=\Le\cup\La\cup\Lb$$
be the decomposition of $\Lab$ into three disjoint sets
\begin{align*}
&\Le\coloneq\{w\in\Lab:s(w)=\e\},\\
&\La\coloneq\left\{w\in\Lab:s(w)\in\Pa\backslash\{\e\}\right\}\text{and}\\
&\Lb\coloneq\left\{w\in\Lab:s(w)\in\Pb\backslash\{\e\}\right\}.
\end{align*}

\begin{remark}\label{w^coor}
If $w\in\Le$ then $w=\widetilde{w}$.
If $w\in\La$ then $w$ and $\widetilde{w}$ differ at most at three coordinates.
If $w\in\Lb$ then $w$ and $\widetilde{w}$ differ at most at two coordinates.
\end{remark}

\begin{lem}\label{PS2.2}
Fix \textcolor{black}{$\b>2$, $\k\in\N_0$ with $\k<\lceil\k/\b+\b\rceil-2$}. Let $\a=\k/\b$.
Let us define $h:\Lab\to\Lab$ by $h(w)=\widehat{w}$ for $w\in\Lab$.
\begin{enumerate}
     \item Let  $w=vu\in\Lab, s(w)=u$. Then $s(\widetilde{w})=\e$ and $\vtx(\widetilde{w})=[0, 0]$.
     \hspace{0.3cm}
     \item The restriction of the mapping $h$ to $\Le$ is one-to-one and satisfies $h(\Le)=\Le$.
     \hspace{0.3cm}
     \item The restriction of the mapping $h$ to $\La$ is one-to-one and satisfies $h(\La)\subset\Le\cup\Lb$.
     \hspace{0.3cm}
     \item Let $p_1\coloneq z^{\a, \b}(b_1)$, where $b_1$ the first character of $b$. The restriction of the mapping $h$ to $\Lb$ is at most $(p_1+3)$-to-one and satisfies $h(\Lb)\subset\Le$.
     \hspace{0.3cm}
     \item The mapping on $\Lab\ni w\mapsto\widetilde{w}$ is at most $(2p_1+7)$-to-one.
\end{enumerate}
\end{lem}

\begin{proof}
(i) This follows from Lemma~\ref{00} and Lemma~\ref{w^}.

(ii) If $u=\e$, then we have $\widehat{w}=w$. Therefore the restriction of $h$ to $\Le$ is one-to-one and $h(\Le)=\Le$.

(iii) Let $\widehat{\Pa}\coloneq\{\widehat{u}:u\in\Pa\backslash\{\e\}\}$. It is easy to check that
\begin{align}\label{Pa^}
\widehat{\Pa}=\{\widehat{a}_1, \widehat{a}_1\widehat{a}_2, \widehat{a}_1\widehat{a}_2\widehat{a}_3,\cdots\}=\{\textcolor{black}{1, 0(\k+1), 0\k(\k+1), 0\k\k(\k+1), \cdots}\}
\end{align}
Let $w, w'\in\La$ and suppose that $\widehat{w}=\widehat{w'}$. 
Let $w=vu, w'=v'u'$ and $u=s(w), u'=s(w')$.
If $|u|=1$, then the last letter of $\widehat{w}$ is $\widehat{u}=\widehat{a}_1=1$. 
Since $\widehat{w}=\widehat{w'}$, the last letter of $\widehat{w'}$ is $\widehat{u}=1$, and the last letter of $\widehat{u'}$ is also $1$.
Therefore $|u'|=1=|u|$ and $u'=u$ by \eqref{Pa^}. \textcolor{black}{This implies that} $w=w'$.

If $|u|>1$, then the first letter of $\widehat{u}$ is $0$ and the last of $\widehat{u}$ is $\textcolor{black}{\k+1}$ and the others of $\widehat{u}$ are $1$. 
\textcolor{black}{Hence} $\widehat{u}=\widehat{u'}$, which means $u=u'$.
Thus the restriction of $h$ to $\La$ is one-to-one.
By Lemma~\ref{w^} (ii) we also have $h(\La)\subset\Le\cup\Lb$.

(iv) Let $\widehat{\Pb}\coloneq\{\widehat{u}:u\in\Pb\backslash\{\e\}\}$.
The list $\widehat{\Pb}$ is ordeblack by increasing length.
The first $p_1+2$ words of $\widehat{\Pb}$ are
$$\widehat{b}_1, \widehat{b}_1b_2, \widehat{b}_1b_2b_3,\cdots,\widehat{b}_1b_2\cdots b_{p_1+2}$$
where $\widehat{b}_1=b_1-1$ or $b_1$. 
On the other hand, if $|\widehat{u}|>p_1+2$, then $\widehat{b}_1=b_1$.

Let $w'\in\{\widehat{w}:w\in\Lab\}$ and $\widehat{s}(w')\in\widehat{\Pb}$ be the largest suffix of $w'$. 
We set $w'=v'u'$, where $u'=\widehat{s}(w')$. 
Suppose that some word $w=vu\in\Lb$, $u=s(w)$ satisfies $\widehat{w}=w'$. 
Assume $|u|>|u'|$. 
Then $\widehat{u}\in\widehat{\Pb}$ implies that $\widehat{s}(w')$ is not maximal. 
This is a contradiction to the definition of $\widehat{s}(w')$. 
Therefore $|u|\le|u'|$ (i.e., $|\widehat{u}|\le|\widehat{s}(w')|$).

Suppose $p_1+3\le|u|<|u'|$. 
Then we have the first letter of $\widehat{u}$ is $b_1$ because $|u|\ge p_1+3$. 
By hypothesis $|u|<|u'|$, it follows that $v=v'c$ for some prefix $c$ of $u'$.
Hence 
$c$ and $u'$ begin with the same character, which is $b_1$.
Also the character following $c$ is the first character of $\widehat{u}$, which is $b_1$. 
Because $u'=c\widehat{u}=\widehat{s}(w')\in\widehat{\Pb}$,  
$cu$ is a prefix of $b$ by Lemma~\ref{PS2.1}, which contradicts the maximality of $u$. 
Thus we have $|u|=|u'|$ or $1\le|u|\le p_1+2$. 

On the other hand, if $k\ge 1$, then there is at most one $w=vu\in\Lab$ such that $w'=\widehat{w}=v\widehat{u}$, $|u|=k$ and $u\in\Pb\backslash\{\e\}$.
Therefore the mapping $w\mapsto\widehat{w}$ is at most $(p_1+3)$-to-one and it follows that the restriction of $h$ 
to $\Lb$ is at most $(p_1+3)$-to-one. 
Furthermore it is easy to check that $h(\Lb)\subset\Le$ by Lemma~\ref{w^} (i).
 
(v) Since $\Lab=\Le\sqcup\La\sqcup\Lb$, we need to consider the restriction of the mapping $w\mapsto\widetilde{w}\coloneq h^2(w)$ to following three sets: $\Le$, $\La$ and $\Lb$.

(Case 1): In Case 1, we consider $h^2|_{\Le}$.
By (ii), $h|_{\Le}$ is one-to-one. 

Since $h(\Le)=\Le$, we have $h^2|_{\Le}=(h|_{\Le})\circ(h|_{\Le})$. 
Therefore $h^2|_{\Le}$ is one-to-one.


(Case 2): In this case we consider $h^2|_{\La}$.
By (iii), $h|_{\La}$ is one-to-one.
 
By Lemma~\ref{w^} $h(\La)\subset\Le\sqcup\Lb$.
Therefore we have 
\begin{align*}
h^2|_{\La}(w)=
\begin{cases}
  &(h|_{\Le})\circ(h|_{\La})(w) \quad\text{if}\,\, w\in h^{-1}(\Le)\cap\La,\\
  &(h|_{\Lb})\circ(h|_{\La})(w) \quad\text{if} \,\,w\in h^{-1}(\Lb)\cap\La.
\end{cases}
\end{align*}
By (ii) and (iv), we have $h|_{\Le}$ is one-to-one and $h|_{\Lb}$ is at most $(p_1+3)$-to-one.
Thus $h^2|_{\La}$ is at most $(p_1+3)$-to-one.



(Case 3): In Case 3, we consider $h^2|_{\Lb}$.
By (iv) we have $h|_{\Lb}$ 
is $(p_1+3)$-to-one. Since $h(\Lb)\subset\Le$ we have 
$$h^2|_{\Lb}(w)=(h|_{\Le})\circ(h|_{\Lb})(w).$$
By (ii), $h|_{\Le}$ is one-to-one and therefore $h^2|_{\Lb}$ is at most $(p_1+3)$-to-one.

Therefore the mapping $w\mapsto\widetilde{w}$ is at most $(2p_1+7)$-to-one. 
\end{proof}

\section{\textcolor{black}{Proof of Theorem~\ref{main}}}\label{S.proof}
\subsection{Functions of bounded total oscillations}\label{bto}
In this section, for more details we refer the reader to \cite{PS1} and \cite{PS3}.
Let $C(\A^\Z)$ be the set of real-valued continuous functions on the full shift $\A^\Z$.
We define for $\bar{f}\in C(\A^\Z)$ and $i\in\Z$,
$$\delta_i(\bar{f})\coloneq\sup\{|\bar{f}(x)-\bar{f}(y)|:x, y\in \A^\Z, \text{where}\, x_k=y_k \,\text{for all}\, k\neq i\}$$
and
$$\|\bar{f}\|_\delta\coloneq\sum_{i\in\Z}\delta_i(\bar{f}).$$
A function is said to have \textit{bounded total oscillations} if $\|\bar{f}\|_\delta<\infty$. 

Let $X\subset \A^\Z$ be a subshift and $C(X)$ be the set of all continuous functions $f:X\to\R$. Every $f\in C(X)$ has a continuous extension $\bar{f}\in C(\A^\Z)$, i.e. $f(x)=\bar{f}(x)$ for all $x\in X$ (see \cite[Proposition 3.2]{PS3}). 
We write $\bar{f}\approx f$.
Also a function $f\in C(X)$ has \textit{bounded total oscillations} if 
$$\|f\|_\delta\coloneq\inf\{\|\bar{f}\|:\bar{f}\in C(A^\Z), f\approx\bar{f}\}<\infty.$$
Let $\mathcal{B}(X)$ be the set of bounded total oscillations with the norm 
$$\| f\|_\delta+\sup_{x\in X}|f(x)|.$$
Then $\mathcal{B}(X)$ is a Banach space (see \cite[Proposition 3.1]{PS3}).
The following lemmas are essential for the argument.
(See \cite{PS1} for more details).

\begin{lem}[{\cite[Lemma 2.4]{PS1}}]\label{PS2.4}
Let $x, y\in X$ and $\Gamma\coloneq\{j\in\Z:x_j\neq y_j\}$. Then for $\Lambda\subset\Z$, $f\in C(X)$ and $\bar{f}\in C(\A^\Z)$ with $\bar{f}\approx f$,
$$\sum_{i\in\Lambda}|f(\s^i x)-f(\s^i y)|\leq\sum_{i\in\Lambda}\sum_{j\in\Gamma}\delta_{j-i}(\bar{f})\leq\operatorname{Card}(\Gamma)\|\bar{f}\|_\delta.$$
\end{lem}


\begin{lem}[{\cite[Lemma 2.5]{PS1}}]\label{PS2.5}
Let $f\in\mathcal{B}(X)$. 
For $\e>0$, there exists $N_\e\in\N$ such that for $m\ge N_\e$, 

\begin{align}\label{7}
\sup\Bigg\{\sum_{1\le i\le m}|f&(\s^i x)-f(\s^i y)|: 
\begin{array}{ll}
x,y\in X, x_k=y_k\\
\text{for all} \,\, k\in\{1,\cdots, m\}
\end{array} 
\Bigg\}\le m\e
\end{align}
and
\begin{align}\label{8}
\sup\Bigg\{\sum_{j\notin\{1,\cdots, m\}}|f&(\s^j x)-f(\s^j y)|: 
\begin{array}{ll}
x,y\in X, x_k=y_k\\
\text{for all} \,\, k\notin\{1,\cdots, m\}
\end{array} 
\Bigg\}\le m\e.
\end{align}

\end{lem}

\subsection{\textcolor{black}{Pressure}}
For $k<l$ we set $[k, l]=\{k, k+1,\cdots, l-1, l\}$. 
Define the projection $J_{[k, l]}:\Sab\to\Lab$ by
$$\Sab\ni x\mapsto J_{[k, l]}(x)\coloneq x_{[k, l]}=x_k x_{k+1}\cdots x_l.$$

For $n\in\N$ we choose a finite set $E^n\subset\Sab$ satisfying that 
\begin{enumerate}
\item[(i)] $J_{[-n, n]}(E^n)=\Lab_{2n+1};$
\item[(ii)] for $x, x'\in E^n$, if $J_{[-n, n]}(x)=J_{[-n, n]}(x')$, then $x=x'$.
\end{enumerate}
For $\varphi\in C(\Sab)$ we set
$$P_{E^n}(\varphi)\coloneq\frac{1}{2n+1}\ln\Xi^n(\varphi)$$
where
$$\Xi^n(\varphi)\coloneq\sum_{x\in E^n}\exp\bigg(\sum_{j\in[-n, n]}\varphi(\s^j x)\bigg).$$
We define the pressure $p(\varphi)$ as
\begin{equation}
p(\varphi)\coloneq\lim_{n\to\infty}P_{E^n}(\varphi).\label{11}
\end{equation}
\color{black}
The map $\Lab\ni w\mapsto g(w)\in\{0, 1\}$ is defined by 
\begin{align*}
g(w)\coloneq
   \begin{cases}
   0 & \quad\text{if}\,\,k_1(w)=0, k_2(w)>0, 
   b_{k_2(w)+1}\textcolor{black}{\in\{0,1,\cdots,\k\}},\\
   \k & \quad\text{otherwise}.
   \end{cases}
\end{align*} 
\textcolor{black}{
We remark that if $vu\in\Lab$ and $s(v)=\e$, then $g(vu)=g(u)$.
}
\begin{lem}\label{g(w)}
Let $w=x_kx_{k+1}\cdots x_{k+m-1}\in\Lab_m$. If we define $y=(\cdots, y_{-1}, y_0, y_1,\cdots)$ by
\begin{align*}
y_j=
   \begin{cases}
   x_j\quad & j=k,\cdots, k+m-1,\\
   g(w)\quad & j=k+m,\\
   \k\quad & \text{otherwise},
   \end{cases}
\end{align*}
then we have $y\in\Sab$.
\end{lem}

\begin{proof}
Let $w=x_k x_{k+1}\cdots x_{k+m-1}\in\Lab_m$. 
Then $w$ can be extended to the left by $\k$, that is, $\textcolor{black}{\k\k\cdots\k} w\in\Lab$.

Also $w$ can be extended to the right by $\textcolor{black}{\k\k\k\cdots}$ or $\textcolor{black}{0\k\k\k\cdots}$.
If $s(w)=\e$, that is, $\vtx(w)=[0, 0]$, this is clear. 
If $w=vu, u=s(w)\in\Pa\backslash\{\e\}$, then $u=a_1\cdots a_p=0\textcolor{black}{\k\cdots\k}$ for some $p\ge 1$. Since $a_i=\k$ for any $i>1$, we get $w\textcolor{black}{\k\k\cdots}\in\Xab$. 

If $w=vu, u=s(w)\in\Pb\backslash\{\e\}$, then $u=b_1\cdots b_p$ for some $p\ge 1$. 
When $k_1(w)=0$ and $b_{p+1}=\textcolor{black}{1, \cdots,\k}$, we may change $b_{p+1}$ into $g(w)=0$ and $\vtx(w0)=[1, 0]$ (see \Eb). 
Hence $w\textcolor{black}{0\k\k\cdots}\in\Xab$. 
When $k_1(w)=0$ and $b_{p+1}\ge \textcolor{black}{\k+1}$, we may change $b_{p+1}$ into $g(w)=\k$ and $\vtx(w1)=[0, 0]$ (see \Ed). Therefore $w\textcolor{black}{\k\k\cdots}\in\Xab$.
 
 When there exist 
 \begin{align*}
 r\ge
 \begin{cases}
 0\quad\text{if}\quad k_1(w)>0,\\
 1\quad\text{if}\quad k_1(w)=0\,\,\text{and}\,\,b_{p+1}=0
  \end{cases}
 \end{align*} 
 such that $b_{p+1}\cdots b_{p+r}=a_{k_1(w)+1}\cdots a_{k_1(w)+r}$ and $b_{p+r+1}\neq a_{k_1(w)+r+1}=1$, we may change $b_{p+r+1}$ into $\k$ and $\vtx(wb_{p+1}\cdots b_{p+r}\k)=[k_1(w)+r+1, 0]$ (see \Eb). \textcolor{black}{Hence}
\begin{align*}
 wg(w)\k\k\cdots&=
 \begin{cases}
 w\k\k\cdots\quad\text{if}\quad k_1(w)>0,\\
 w0\k\k\cdots\quad\text{if}\quad k_1(w)=0\,\,\text{and}\,\,b_{p+1}=0
  \end{cases}\\
  &\in\Xab.
 \end{align*} 
This proves that $y\in\Sab$.
\end{proof}
\color{black}
It is known that $p(\varphi)$ is independent of the choice of  $E^n$ (see \cite[Theorem9.6(i)]{Wa}). From now on we choose
\begin{align}\label{En}
E^n=\left\{x\in\A^\Z\left|
\begin{array}{l}
x_{[-n, n]}\in\Lab,\\
x_{n+1}=g(x_{[-n, n]}),\\
x_j=\k\quad\text{if}\,\,\,j\not\in[-n, n+1]
\end{array}\right.\right\}.
\end{align}
By Lemma \ref{g(w)} we have $E^n\subset\Sab$.
\color{black}

\color{black}

\color{black}

\color{black}

\begin{lem}\label{PS2.9a}
Let $k,l\in\Z$, $m=l-k+1$ and $w\in\Lab_m$. Define $w^\sharp\in\Sab$ by
\begin{eqnarray}
w_j^\sharp
\coloneq
  \begin{cases}
     w_{j-k+1} & \quad\text{if}\,\,\,k\le j\le l \\
     g(w) & \quad\text{if}\,\,\,j=l+1 \\
     \k & \quad\text{otherwise}.
  \end{cases}\label{16}
\end{eqnarray}
Then we have
\begin{enumerate}
   \item The configurations $w^\sharp$ and $\widetilde{w}^\sharp$ differ at most at four coordinates.
   \item The map $\Lab_m\ni w\mapsto\widetilde{w}^\sharp$ is at most $(\M)$-to-one.
\end{enumerate}
\end{lem}

\begin{proof}
(i) The configurations $w$ and $\widetilde{w}$ differ at most three coordinates by Remark~\ref{w^coor}. \textcolor{black}{Consequently,} the configurations $w^\sharp$ and $\widetilde{w}^\sharp$ differ at most at four coordinates including the coordinate of $g(w)$.

(ii) It is clear that the map $\Lab_m\ni w\mapsto w^\sharp$ is one-to-one. Therefore the map $\Lab_m\ni w\mapsto\widetilde{w}\mapsto\widetilde{w}^\sharp$ is at most $(\M)$-to-one by Lemma~\ref{PS2.2}.
\end{proof}

By using Lemma~\ref{PS2.9a} we can prove the following lemma which is similar to \cite[Lemma 2.9]{PS1}.

\begin{lem}(cf. \cite[Lemma 2.9]{PS1})\label{PS2.9}
Let $k,l\in\Z$, $m=l-k+1$ and $w\in\Lab_m$. Let $w^\sharp\in\Sab$ be defined as in \eqref{16}.
Then for $\varphi$ with $\|\varphi\|_\delta<\infty$ the pressure $p(\varphi)$ is equal to
$$\lim_{m\to\infty}\frac{1}{m}\ln\quad\smashoperator{\sum_{w\in\Lab_m:s(w)=\e}}\quad\exp\sum_{j=k}^l\varphi(\s^j w^\sharp).$$
\end{lem}

In other words the following statement holds: 
For any $\e>0$ and $\varphi\in\B(\Sab)$,  
there exist $N_{\e, \varphi}$ such that for $m\ge N_{\e, \varphi}$,
\begin{align}\label{23}
e^{-m\e}\le\frac{e^{mp(\varphi)}}{\sum_{\substack{w\in\Lab_m:\\s(w)=\e}}\exp\sum_{j=k}^l\varphi(\s^jw^\sharp)}\le e^{m\e}.
\end{align}

\color{black}
\begin{proof}[\bf{Sketch of the proof}]
In Lemma~\ref{PS2.9a}, we modify the definition of $w^\sharp$ from \cite[Lemma 2.9]{PS1} to obtain $w^\sharp\in E^n$ defined by \eqref{En}. 
Taking into account the replacement of $\widehat{w}^\sharp$ by $\widetilde{w}^\sharp$, we obtain the following lower and upper bounds for the ratio:
\begin{align*}
e^{-4\|\varphi\|_\delta}\le\frac{\sum_{w\in\Lab_m}\exp\sum_{j=k}^l\varphi(\s^jw^\sharp)}{\sum_{w\in\Lab_m:s(w)=\e}\exp\sum_{j=k}^l\varphi(\s^jw^\sharp)}\le e^{4\|\varphi\|_\delta}(\M)
\end{align*}
by using Lemma~\ref{PS2.9a}. The desiblack result follows from taking $\frac{1}{m}\ln$ and the limit $m\to\infty$.
\end{proof}

\color{black}

\subsection{Upper and lower bounds.}\label{3.1}
Let $\varphi\in\B(\Sab)$. 
In \textcolor{black}{this section}, we establish upper and lower bounds for $\nu([y_0\cdots y_{m-1}])$ for any equilibrium measure $\nu$ of $\varphi$. Theorem~\ref{main} is proved in \S\ref{3.2}.

\begin{defi}
An invariant probability measure $\nu$ is a \textit{tangent functional to the pressure $p$ at $\varphi$} if
$$p(\varphi+f)\ge p(\varphi)+\int f\,\,d\nu\quad\text{for all continuous}\, f.$$
We denote as $\partial p(\varphi)$ the set of tangent functionals to the pressure at $\varphi$.
\end{defi}
Assume that \textcolor{black}{$\b>2$ and $\k\in\N_0$ with $\kappa<\lceil\k/\b+\b\rceil-2$, $\a=\k/\b$}. 
We first consider the case of a unique tangent functional $\nu$ to the pressure at $\varphi$ and then extend this result to any $\varphi\in\B(\Sab)$.
(see \cite{PS1} and the references given there).
Let $\bar{u}\in\Lab_m$ be fixed and set
\begin{align*}
I_{\bar{u}}(y)\coloneq
   \begin{cases}
   1 & \,\,\text{if}\quad y_0\cdots y_{m-1}=\bar{u},\\
   0 & \,\,\text{otherwise}.
   \end{cases}
\end{align*}
Then it follows that
\begin{align*}
\nu(I_{\bar{u}})
& =\frac{d}{dt}\lim_{n\to\infty}P_{E^n}(\varphi+tI_{\bar{u}})\bigg|_{t=0}=\lim_{n\to\infty}\frac{d}{dt}P_{E^n}(\varphi+tI_{\bar{u}})\bigg|_{t=0}\\
& =\lim_{n\to\infty}\frac{1}{2n+1}\sum_{j=-n}^n\frac{\sum\nolimits_{x\in E^n}I_{\bar{u}}(\s^jx)\exp\sum\nolimits_{i=-n}^n\varphi(\s^ix)}{\sum\nolimits_{x\in E^n}\exp\sum\nolimits_{i=-n}^n\varphi(\s^ix)}\\
& =\lim_{n\to\infty}\frac{1}{2n+1}\sum_{[k, l]\subset[-n, n]}\frac{\Xi_{[k, l]}^n(\bar{u})}{\sum\nolimits_{v\in\Lab_m}\Xi_{[k, l]}^n(v)}
\stepcounter{equation}\tag{\theequation}\label{17}
\end{align*}
(see \cite{PS1} and the references given there).
\color{black}
Let $[k, l]\subset[-n, n], x\in E^n$. Set
$$x_k^-\coloneq x_{(-\infty, k-1]}\quad\text{and}\quad x_l^+\coloneq x_{[l+1, \infty)}.$$
Because of our choice of $E^n$, the definition \eqref{decom} and \eqref{hat} can be extended to $x_k^-$ since $x_j=\k$ for $j<-n$. 
Let $s(x_k^-)\in\P$ be a largest suffix of $x_k^-$, 
we set
$$\widehat{x_k^-}\coloneq y\widehat{s(x_k^-)}\quad\text{where}\quad x_k^-=ys(x_k^-)$$
and
$$\widetilde{x_k^-}\coloneq\widehat{(\widehat{x_k^-})}\coloneq y'\widehat{s(\widehat{x_k^-})}
\quad\text{where}\quad\widehat{x_k^-}=y's(\widehat{x_k^-}).$$
Let $v\in\Lab_{l-k+1}$. We define
$$E_{[k, l]}^n(v)\coloneq\{x\in E^n:x_{[k, l]}=v\},$$
$$E_{[k, l]}^{\ast, n}(v)\coloneq\{x\in E_{[k, l]}^n(v):s(x_k^-)=\e\}$$
and 
$$\Xi_{[k, l]}^n(v)\coloneq\sum_{x\in E_{[k, l]}^n(v)}\exp\bigg(\sum_{j\in[-n, n]}\varphi(\s^jx)\bigg),$$
$$\Xi_{[k, l]}^{\ast, n}(v)\coloneq\sum_{x\in E_{[k, l]}^{\ast, n}(v)}\exp\bigg(\sum_{j\in[-n, n]}\varphi(\s^jx)\bigg).$$
We have
$$\Xi^n(\varphi)=\sum_{v\in\Lab_{l-k+1}}\Xi_{[k, l]}^n(v).$$

\color{black}
We now consider the upper and lower bounds for $\Xi_{[k, l]}^{n}(v)$. 
In the case of $\b$-shifts, to obtain bounds for $\Xi_{[k, l]}^{n}(v)$ in terms of $E_{[k, l]}^{\ast, n}(\widehat{v})$, Pfister and Sullivan introduced two bounded-to-one maps $\widehat{f}:E_{[k, l]}^n(v)\to E_{[k, l]}^{\ast, n}(\widehat{v})$ and $f:E_{[k, l]}^{\ast, n}(\widehat{v})\to E_{[k, l]}^{\ast, n}(v)$  (see \cite[Lemma 2.7, 2.8]{PS1}). 
Since the function $\widetilde{w}$ as in \eqref{hat} takes the role of $\widehat{w}$ in \cite{PS1}, we define new corresponding maps $\widetilde{f}$ and $f$ by replacing $\widehat{v}$ with $\widetilde{v}$ in the definitions of $\widehat{f}$ and $f$. 
However, the images of such maps do not necessarily belong to $E^n$. 
Thus we should also change the $(n+1)$-th coordinate of $x$ by using the map $g$ (Lemma~\ref{g(w)}). 
In Lemmas~\ref{PS2.7a} and \ref{PS2.8a} we give precise definitions of $\widetilde{f}$ and $f$ and describe their properties.
\color{black}

\color{black}
\begin{lem}\label{PS2.7a}
Let $[k, l]\subset[-n, n]$ and $v\in\Lab_{l-k+1}$.
Define a map $\widetilde{f}$ as follows: 
\begin{align*}
\widetilde{f}:E_{[k, l]}^n(v)\to E_{[k, l]}^{\ast, n}(\widetilde{v}),\quad x=x_k^-vx_l^+\mapsto \widetilde{f}(x)\coloneq \widetilde{x_k^-}\widetilde{v}{x'}_l^+
\end{align*}
where ${x'}_l^+=({x'}_{l+1}, {x'}_{l+2}, \cdots)$ defined by 
\begin{align*}
{x'}_j\coloneq
\begin{cases}
     x_j & \quad j\neq n+1,\\
     g(x_{l+1}\cdots x_n) & \quad j=n+1.
\end{cases}
\end{align*}
Then the following statements hold:
\begin{enumerate}
    \item The map $\widetilde{f}$ is at most $(\M)$-to-one.
    \item The sequences $x$ and $\widetilde{f}(x)$ differ at most at seven coordinates.
\end{enumerate}
\end{lem}

\begin{proof}
Since 
$\vtx(x_k^-)=\vtx(\widetilde{v})=[0, 0]$
 (see Lemmas~\ref{00} and \ref{w^}), $\widetilde{f}(x)$ is well defined and $\widetilde{f}(x)\in E_{[k, l]}^{\ast, n}(\widetilde{v})$.

(i) By Lemma \ref{PS2.2} the map $x_k^-\mapsto\widetilde{x_k^-}$ is at most $(\M)$-to-one. Therefore the map $\widetilde{f}$ is at most $(\M)$-to-one because $v$  and $\widetilde{v}$ are fixed and $x_{n+1}$ is uniquely determined by the choice of $x_k^-$.

(ii) By Remark~\ref{w^coor}, $x_k^-$ and $v$ differ from $\widetilde{x_k^-}$ and $\widetilde{v}$ at most at three coordinates respectively.
Thus the sequences $x$ and $\widetilde{f}(x)$ differ at most at seven coordinates.
\end{proof}


\begin{lem}(cf. \cite[Lemma 2.7]{PS1})\label{PS2.7}
Let $[k, l]\subset[-n, n], v\in\Lab_{l-k+1}$ and $\varphi\in\B(\Sab)$. 
Then
$$\Xi_{[k, l]}^{\ast, n}(v)\le\Xi_{[k, l]}^n(v)\le(\M)e^{7\|\varphi\|_\delta}\Xi_{[k, l]}^{\ast,n}(\widetilde{v}).$$
\end{lem}

\color{black}
\begin{proof}[\bf{Sketch of the proof}]
The upper bound of the formula in \cite[Lemma 2.7]{PS1} is controlled by two factors: the maximum number of preimages under the map $\widehat{f}$, which is $z^\b(c_1)+2$, and the difference in the number of characters between $x$ and $\widehat{f}(x)$, which is at most $2$.
In order to replace $\widehat{f}$ with $\widetilde{f}$, we should substitute the term $z^\b(c_1)+2$ by $\M$ 
(Lemma~\ref{PS2.7a}(i)).
Moreover, we change the exponent of $e$ from $2\|\varphi\|_\delta$ to $7\|\varphi\|_\delta$ because the sequences $x$ and $\widetilde{f}(x)$ differ at most at seven coordinates (Lemma~\ref{PS2.7a}(ii)).
\end{proof}

\color{black}

\color{black}

\color{black}
\begin{lem}\label{PS2.8a}
Let $[k, l]\subset[-n, n]$ and $v\in\Lab_{l-k+1}$. By the definition of $\Zab(v)$ there exist a $\overline{v}=\overline{v}_1\cdots\overline{v}_{|\overline{v}|}\in\Lab$ with $|\overline{v}|=\Zab(v)+2$ and $\vtx(v\overline{v})=[0, 0]$.
For such $\overline{v}$ we define a map $f:E_{[k, l]}^{\ast, n}(\widetilde{v})\to E_{[k, l]}^{\ast, n}(v)$, 
$$f(x_k^-\widetilde{v}x_l^+)\coloneq x_k^-v{x'}_l^+,$$
where ${x'}_l^+$ is defined as
\begin{align*}
\Zab(v)&+2\le n-l\\
& \Longrightarrow {x'}_i\coloneq
\begin{cases}
     \overline{v}_{i-l} & \quad i=l+1,\cdots, l+\Zab(v)+2,\\
     x_i & \quad i>l+\Zab(v)+2 \,\,\text{and}\,\, i\neq n+1,\\
     g(v\bar{v}x_{j}\cdots x_n) & \quad i=n+1\,\,\text{where}\,\,j=l+\Zab(v)+3,
    \end{cases}\\
\Zab(v)&+2>n-l\\
& \Longrightarrow {x'}_i\coloneq
\begin{cases}
     \overline{v}_{i-l} & \quad i=l+1,\cdots, n,\\
     g(v\bar{v}_1\cdots\bar{v}_{n-l}) & \quad i=n+1,\\
     \k & \quad i>n+1.\\
    \end{cases}
\end{align*}
Then we have
\begin{enumerate}
    \item The sequences $x_k^-\widetilde{v}x_l^+$ and $x_k^-\widetilde{v}{x'}_l^+$ differ at most at $\Zab(v)+3$ coordinates.
    \item The sequences $x_k^-\widetilde{v}{x'}_l^+$ and $x_k^-v{x'}_l^+$ differ at most at three coordinates.
    \item The map $f$ is at most $|\A|^{\Zab(v)+3}$-to-one.
\end{enumerate}
\end{lem}

\begin{proof}
(i) Since the $i$-th coordinate of $x_k^-\widetilde{v}x_l^+$ is equal to that of $x_k^-\widetilde{v}{x'}_l^+$ for any $i$ except for $i=l+1,\cdots, l+\Zab(v)+2, n+1$, hence $x_k^-\widetilde{v}x_l^+$ and $x_k^-\widetilde{v}{x'}_l^+$ differ at most at $\Zab(v)+3$ coordinates.

(ii) This is clear from Remark~\ref{w^coor}.

(iii) Since $v$ and $\widetilde{v}$ are fixed we need only consider the configurations of $x_l^+$ and ${x'}_l^+$. So by (i) the map $f$ is at most $|\A|^{\Zab(v)+3}$-to-one.
\end{proof}

Lemma~\ref{PS2.8a} shows some lower estimate similar to \cite[Lemma 2.8]{PS1}.

\begin{lem}(cf. \cite[Lemma 2.8]{PS1})\label{PS2.8}
Let $[k, l]\subset[-n, n], v\in\Lab_{l-k+1}$ and $\varphi\in\B(\Sab)$. 
Then
$$\Xi_{[k, l]}^{\ast, n}(v)\ge |\A|^{-(\Zab(v)+3)}e^{-(\Zab(v)+6)\|\varphi\|_\delta}\Xi_{[k, l]}^{\ast,n}(\widetilde{v}).$$
\end{lem}

\color{black}
\begin{proof}[\bf{Sketch of the proof}]
In the proof of \cite[Lemma 2.8]{PS1}, the lower bound is determined by the maximum number of preimages under the map $f$, which is $|\A|^{(z^\b(v)+1)}$, and the number of coordinates where $x$ and $f(x)$ differ, which is at most $z^\b(v)+2$.
To modify the definition of $f$ as in Lemma~\ref{PS2.8a}, we should change $|\A|^{(z^\b(v)+1)}$ to $|\A|^{(\Zab(v)+3)}$ (Lemma~\ref{PS2.8a}(iii)) and replace $(z^\b(v)+2)\|\varphi\|_\delta$ with $(\Zab(v)+6)\|\varphi\|_\delta$ by taking into account the difference in the number of characters between $x$ and $f(x)$ (Lemma~\ref{PS2.8a}(i), (ii)).
\end{proof}

\color{black}

\color{black}
\begin{lem}\label{PS3.1}
Let $\bar{u}\in\Lab_m$. 
For any $\e>0$ and $\varphi\in\B(\Sab)$ 
there exists $N_{\e, \varphi}$ so that if 
$J_{[1, m]}(y)=\bar{u}$ and $m\ge N_{\e, \varphi}$, then
$$\nu(I_{\bar{u}}) \le K_{\varphi, \e}^{+}(m, \b, \k)\exp\biggl(\sum_{j=1}^{m}\varphi(\s^j y)-mp(\varphi)\biggr),$$
where
$$K_{\varphi, \e}^{+}(m, \b, \k)=(\M)e^{11\|\varphi\|_\delta}e^{5m\e}.$$
$N_{\e, \varphi}$ chosen so that all of the inequalities \eqref{7}, \eqref{8} and \eqref{23} are satisfied for $m\ge N_{\e, \varphi}$.
\end{lem}

\begin{proof}
Let $\e>0$, $\varphi\in\B(\Sab)$, and let $N_{\e, \varphi}$ be defined as above. We consider the reciprocal of the term in the sum \eqref{17} with $j=k$ and $[k, l]\subset[-n, n]$. By Lemma~\ref{PS2.7},
\begin{align*}
\frac{\Xi^n(\varphi)}{\Xi_{[k, l]}^n(\bar{u})}
& =\frac{\sum\nolimits_{v\in\Lab_m}\Xi_{[k, l]}^n(v)}{\Xi_{[k, l]}^n(\bar{u})}
\ge\frac{\sum\nolimits_{v:s(v)=\e}\Xi_{[k, l]}^{\ast, n}(v)}{\Xi_{[k, l]}^n(\bar{u})}\\
& \ge\frac{1}{(\M)e^{7\|\varphi\|_\delta}}\sum_{v:s(v)=\e}\frac{\Xi_{[k, l]}^{\ast, n}(v)}{\Xi_{[k, l]}^{\ast, n}(\widetilde{\bar{u}})}.
\end{align*}

Since $s(v)=\e$ and $s(\widetilde{\bar{u}})=\e$, we have
\begin{align}\label{19}
\mathcal{W}
& =\{(x_k^-, x_l^+)|x_k^-vx_l^+\in E_{[k, l]}^{\ast, n}(v)\}
=\{(x_k^-, x_l^+)|x_k^-\widetilde{\bar{u}}x_l^+\in E_{[k, l]}^{\ast, n}(\widetilde{\bar{u}})\}.
\end{align}
If we set $\Xi_{[k, l]}^{\ast, n}(v)=\sum_{w\in\mathcal{W}}a_w$ and $\Xi_{[k, l]}^{\ast, n}(\widetilde{\bar{u}})=\sum_{w\in\mathcal{W}}b_w$, then
\begin{align}\label{24}
\inf_{w\in\mathcal{W}}\frac{a_w}{b_w}\le\frac{\sum_{w\in\mathcal{W}}a_w}{\sum_{w\in\mathcal{W}}b_w}=\frac{\sum_{w\in\mathcal{W}}(a_w/b_w)b_w}{\sum_{w\in\mathcal{W}}b_w}\le\sup_{w\in\mathcal{W}}\frac{a_w}{b_w}.
\end{align}
Hence we have
\begin{align}\label{inf}
\frac{\Xi_{[k, l]}^{\ast, n}(v)}{\Xi_{[k, l]}^{\ast, n}(\widetilde{\bar{u}})}
\ge\inf_{(x_k^-, x_l^+)\in\mathcal{W}}\exp\left(\sum_{j=-n}^n\varphi(\s^j(x_k^-vx_l^+))-\varphi(\s^j(x_k^-\widetilde{\bar{u}}x_l^+))\right).
\end{align}
Let $v^\sharp$, ${\widetilde{\bar{u}}}^\sharp$ be defined as in \eqref{16} and $y\in\Sab$ with $J_{[1, m]}(y)=\bar{u}$. Then $J_{[k, l]}(\s^{-k+1}y)=\bar{u}=J_{[k, l]}({\bar{u}}^\sharp)$ and  
\begin{align*}
\sum_{j=-n}^n&(\varphi(\s^j(x_k^+vx_l^+))-\varphi(\s^j(x_k^+\widetilde{\bar{u}}x_l^+)))\\
& =\sum_{j\in[-n, k)\cup(l, n]}(\varphi(\s^j(x_k^+vx_l^+))-\varphi(\s^j(x_k^+\widetilde{\bar{u}}x_l^+)))\\
& \quad+\sum_{j\in[k, l]}(\varphi(\s^j(x_k^+vx_l^+))-\varphi(\s^jv^\sharp))+\sum_{j\in[k, l]}\varphi(\s^jv^\sharp)-\sum_{j\in[1, m]}\varphi(\s^jy)\\
& \quad+\sum_{j\in[k, l]}(\varphi(\s^j\s^{-k+1}y)-\varphi(\s^j\bar{u}^\sharp))+\sum_{j\in[k, l]}(\varphi(\s^j\bar{u}^\sharp)-\varphi(\s^j{\widetilde{\bar{u}}}^\sharp))\\
& \quad\quad+\sum_{j\in[k, l]}(\varphi(\s^j{\widetilde{\bar{u}}}^\sharp)-\varphi(\s^j(x_k^+\widetilde{\bar{u}}x_l^+))).
\end{align*}
By Lemma~\ref{PS2.4} and Lemma~\ref{PS2.5}, if $m\ge N_{\e, \varphi}$, then
\begin{align}
& \left|\sum_{j=-n}^n(\varphi(\s^j(x_k^+vx_l^+))-\varphi(\s^j(x_k^+\widetilde{\bar{u}}x_l^+)))-\left(\sum_{j\in[k, l]}\varphi(\s^jv^\sharp)-\sum_{j\in[1, m]}\varphi(\s^jy)\right)\right|\nonumber\\
& \le4m\e+4\|\varphi\|_\delta.\label{ine}
\end{align}

So we get for $[k, l]\subset[-n, n]$,
\begin{align*}
\sum_{v:s(v)=\e}\frac{\Xi_{[k, l]}^{\ast, n}(v)}{\Xi_{[k, l]}^{\ast, n}(\widetilde{\bar{u}})}
& \ge e^{-4m\e-4\|\varphi\|_\delta}\exp\left(-\sum_{j=1}^m\varphi(\s^jy)\right)\sum_{v:s(v)=\e}\exp\left(\sum_{j=k}^l\varphi(\s^jv^\sharp)\right)\\
& \ge e^{-5m\e-4\|\varphi\|_\delta}\exp\left(-\sum_{j=1}^m\varphi(\s^jy)+mp(\varphi)\right)
\end{align*}
by \eqref{inf}, \eqref{ine} and \eqref{23}.
Therefore
\begin{align*}
\frac{\sum_v\Xi_{[k, l]}^n(v)}{\Xi_{[k, l]}^n(\bar{u})}
\ge\frac{e^{-5m\e}}{(\M)e^{11\|\varphi\|_\delta}}\exp\left(-\sum_{j=1}^m\varphi(\s^jy)+mp(\varphi)\right)
\end{align*}
and the result follows from \eqref{17} by taking the limit $n\to\infty$.
\end{proof}

\begin{lem}\label{PS3.2}
Let $\bar{u}\in\Lab_m$. For any $\e>0$ and $\varphi\in\B(\Sab)$ 
there exists $N_{\e, \varphi}$ so that if 
$J_{[1, m]}(y)=\bar{u}$ and $m\ge N_{\e, \varphi}$, then
$$\nu(I_{\bar{u}}) \ge K_{\varphi, \e}^{-}(m, \b, \k, \bar{u})\exp\biggl(\sum_{j=1}^{m}\varphi(\s^j y)-mp(\varphi)\biggr),$$
where
$$K_{\varphi, \e}^{-}(m, \b, \k, \bar{u})=\frac{|\A|^{(-\Zab(\bar{u})+3)}e^{-(\Zab(\bar{u})+5)\|\varphi\|_\delta}}{(\M)^2e^{5m\e+8\|\varphi\|_\delta}}$$
$N_{\e, \varphi}$ chosen so that all of the inequalities \eqref{7}, \eqref{8} and \eqref{23} are satisfied for $m\ge N_{\e, \varphi}$.
\end{lem}

\begin{proof}
Let $\e>0$ and $\varphi$ be given and $N_{\e, \varphi}$ defined as above.
We consider \eqref{17} with $j=k$ and $[k, l]\subset[-n, n]$.
By Lemma~\ref{PS2.2}, \ref{PS2.7} and \ref{PS2.8},
\begin{align*}
\frac{\Xi^n(\varphi)}{\Xi_{[k, l]}^n(\bar{u})}
& =\frac{\sum\nolimits_{v\in\Lab_m}\Xi_{[k, l]}^n(v)}{\Xi_{[k, l]}^n(\bar{u})}
\le\frac{\sum\nolimits_{v\in\Lab_m}\Xi_{[k, l]}^n(v)}{\Xi_{[k, l]}^{\ast, n}(\bar{u})}\\
& \le\frac{(\M)e^{7\|\varphi\|_\delta}}{|\A|^{(-\Zab(\bar{u})+3)}e^{-(\Zab(\bar{u})+5)\|\varphi\|_\delta}}\sum_{v\in\Lab_m}\frac{\Xi_{[k, l]}^{\ast, n}(\widetilde{v})}{\Xi_{[k, l]}^{\ast, n}(\widetilde{\bar{u}})}\\
& \le\frac{(\M)^2e^{7\|\varphi\|_\delta}}{|\A|^{(-\Zab(\bar{u})+3)}e^{-(\Zab(\bar{u})+5)\|\varphi\|_\delta}}\sum_{v:s(v)=\e}\frac{\Xi_{[k, l]}^{\ast, n}(v)}{\Xi_{[k, l]}^{\ast, n}(\widetilde{\bar{u}})}
\end{align*}
As in the proof of Lemma~\ref{PS3.1}, if $s(v)=\e$, then we can estimate the ratio $\Xi_{[k, l]}^{\ast, n}(v)/\Xi_{[k, l]}^{\ast, n}(\widetilde{\bar{u}})$ from above by replacing the infimum in \eqref{inf} by the supremum since \eqref{19} and \eqref{24} hold. Let $y\in\Sab$ be $J_{[1, m]}(y)=\bar{u}$. If $m\ge N_{\e, \varphi}$,  then
\begin{align*}
\sum_{v:s(v)=\e}\frac{\Xi_{[k, l]}^{\ast, n}(v)}{\Xi_{[k, l]}^{\ast, n}(\widetilde{\bar{u}})}
& \le e^{5m\e+4\|\varphi\|_\delta}\exp\left(-\sum_{j\in[1, m]}\varphi(\s^jy)+mp(\varphi)\right)
\end{align*}
by \eqref{ine} and \eqref{23}. Hence
\begin{align*}
\frac{\sum_v\Xi_{[k, l]}^n(v)}{\Xi_{[k, l]}^n(\bar{u})}
& \le\frac{(\M)^2e^{5m\e+8\|\varphi\|_\delta}}{|\A|^{(-\Zab(\bar{u})+3)}e^{-(\Zab(\bar{u})+5)\|\varphi\|_\delta}}\exp\left(\sum_{j=1}^m-\varphi(\s^jy)+mp(\varphi)\right).
\end{align*}
The result follows from \eqref{17} by taking the limit $n\to\infty$.
\end{proof}

By the same argument as in the proof of Pfister and Sullivan,  Lemmas~\ref{PS3.1} and \ref{PS3.2} can be extended to $\varphi$ for which the uniqueness of equilibrium measures does not hold. So the proofs are omitted.

\def\Zabn{\bar{z}^{\a, \b}(n)}

\subsection{Proof of Theorem~\ref{main}}\label{3.2}
The proof is similar to the proof of Pfister and Sullivan for \cite[Theorem 2.12]{PS1}. We can prove the theorem similarly using Lemma~\ref{PS3.1} and \ref{PS3.2}. If $\lim_{n\to\infty}\Zabn/n=0$, then $m^{-1}\ln K_{\varphi, \e}^-(m, \beta, \bar{u})>-6\e$ for sufficiently large $m$. Since $m^{-1}\ln K_{\varphi, \e}^+(m, \beta)<6\e$ for large $m$ the inequality \eqref{weakGibbs} holds. 

On the other hand, if $\limsup_{n\to\infty}\Zabn/n>0$, then there exist a sequence $\{m_k\}_k$ and $w^k\in\Lab_{m_k}$ such that $\lim_k\Zab(w^k)/m_k=a>0$. If we set 
\begin{align*}
\underline{w}^k
\coloneq&\,w^k a_{k_1(w^k)+1}\cdots a_{k_1(w^k)+\Zab(w^k)}\\
=&
\begin{cases}
w^k\textcolor{black}{0\k\cdots\k}\quad\text{if}\,\, k_1(w^k)=0,\\
w^k\textcolor{black}{\k\k\cdots\k}\quad\text{if}\,\, k_1(w^k)>0,
\end{cases}
\end{align*}
then $\nu(I_{w^k})=\nu(I_{\underline{w}^k})$ and $|\underline{w}^k|=m_k+\Zab(w^k)$.
If $\Zab(w^k)$ is large enough and $y\in[w^k]$, then
\begin{align*}
&\left|\sum_{j=1}^{\Zab(w^k)}\varphi(\s^{m_k+j}y)-\Zab(w^k)\varphi(\k)\right|\\
&\le\left|\sum_{j=2}^{\Zab(w^k)}\left(\varphi(\s^{m_k+j}y)-\varphi(\k)\right)\right|+\left|\varphi(\s^{m_k+1}y)-\varphi(\k)\right|\\
&\le\e\Zab(w^k)+2\|\varphi\|\le2\e\Zab(w^k),
\end{align*}
where $\k$ is the all-$\k$ configuration.
Let $x\in\Sab$ be $J_{[1, m_k]}(x)=w^k$. 
As in the proof of \cite[Theorem 2.12]{PS1} we can deduce from $\nu(I_{w^k})=\nu(I_{\underline{w}^k})$ that 
\begin{align*}
\lim_{k\to\infty}&\frac{1}{m_k}\left(\ln\nu(I_{w^k})-\sum_{j=1}^{m_k}\psi(\s^jx)\right)\\
&\le\lim_{k\to\infty}\frac{\Zab(w^k)}{m_k}\underbrace{(7\e+\varphi(\k)-p(\varphi))}_{<0\,\,\text{if}\,\,\e\,\,\text{is small enough}}<0.
\end{align*}

\color{black}
\subsection*{Acknowledgement}
This work was supported by JST SPRING, Grant Number JPMJSP2127.

\color{black}


\begin{thebibliography}{99}
\bibitem[Bu]{Bu} \'{J}er\^{o}me Buzzi. Specification on the interval, \textit{Trans. Am. Math. Soc.} \textbf{349} (1997), no. 7, 2737-2754.
\bibitem[Ca]{Ca} D. Capocaccia. A definition of Gibbs state for a compact set with a $\Z^\nu$ action. \textit{Comm. Math. Phys.} \textbf{48} (1976), 85-88.
\bibitem[CLR]{CLR} Leonard Carapezza, Marco L\'{o}pez, and Donald Robertson. Unique equilibrium states for some intermediate beta transformations \textit{Stoch. Dyn} \textbf{21} (2021), no.6, 25 (English), Id/No 2150035.
\bibitem[CY]{CY} Yong Moo Chung and Kenichiro Yamamoto.  Large deviation principle for piece-wise monotonic maps with density of periodic measures, \textit{Ergodic Theory Dyn. Syst.} \textbf{43} (2023), no. 3, 861-872.
\bibitem[FP]{FP} B. Faller and C.-E. Pfister. A point is normal for almost all maps $\beta x+\alpha$ mod 1 or generalized $\beta$-transformations. \textit{Ergod. Th.} \& \textit{Dynam. Sys.} \textbf{29} (2009), 1529-1547.
\bibitem[Ho]{Ho2} F. Hofbauer. Maximal Measures for Simple Piecewise Monotonic Transformations. \textit{Z. Wahrscheinlichkeitstheorie verw. Gebiete.} \textbf{52} (1980), 289-300.
\bibitem[HR]{HR} N. T. A. Haydn and D. Ruelle. Equivalence of Gibbs and equilibrium states for homeomorphisms satisfying expansiveness and specification. \textit{Comm. Math. Phys.} \textbf{148} (1992), 155-167.
\bibitem[OS]{OS} Mai Oguchi and Mao Shinoda. Hausdorff dimension of the parameters for $(\a, \b)$-transformations with the specification property. \textit{Dyn. Syst.} \textbf{39} (2024), no. 4, 848-855.
\bibitem[Pa1]{Pa1} W. Parry. On the $\beta$-expansions of real numbers. \textit{Acta Mathematica Hungarica.} \textbf{11} (1960), no. 3-4, 401-416.
\bibitem[Pa2]{Pa2} W. Parry. Representations for real numbers. \textit{Acta Mathematica Academiae Scientiarum Hungarica} \textbf{15} (1964), 95-105.
\bibitem[PS1]{PS2} C.-E. Pfister and W. G. Sullivan. Weak Gibbs measures and large deviations. \textit{Nonlinearity} \textbf{31} (2018), 49-53.
\bibitem[PS2]{PS3} C.-E. Pfister and W. G. Sullivan. Asymptotic decoupling and weak Gibbs measures for finite alphabet specification property. \textit{Nonlinearity} \textbf{33} (2020), 4799-4817.
\bibitem[PS3]{PS1} C.-E. Pfister and W. G. Sullivan. Equilibrium measures of the natural extension of $\beta$-shifts. \textit{Ergod. Th.} \& \textit{Dynam. Sys}. \textbf{42} (2022), 2415-2430.
\bibitem[Re]{Re} A. R\'enyi. Representations for real numbers and their ergodic properties. \textit{Acta Math. Hungar.} \textbf{8} (1957), 477-493.
\bibitem[RS]{RS} Peter Raith and Angela Stachelberger.  Topological transitivity for a class of monotonic mod one transformations. \textit{Aequationes Math.} \textbf{82} (2011), no. 1-2, 91-109.
\bibitem[Sc]{Sc} J. Schmeling. Symbolic dynamics for $\b$-shifts and self-normal numbers. \textit{Ergod. Th.} \& \textit{Dynam. Sys}. \textbf{17} (1997), 675-694.
\bibitem[Wa]{Wa} P. Walters. \textit{An Introduction to Ergodic Theory.} Springer, New York, (1982).
\bibitem[Yu1]{Yu1} M. Yuri. Weak Gibbs measures for intermittent systems and weakly Gibbsian
states in statistical mechanics. \textit{Comm. Math. Phys.} \textbf{241} (2003).
\bibitem[Yu2]{Yu2} M. Yuri. Weak Gibbs measures and the local product structure. \textit{Ergodic Theory Dynam. Systems} \textbf{22} (2002), no. 6, 1933-1955.
\bibitem[Yu3]{Yu3} M. Yuri. Multifractal analysis of weak Gibbs measures for intermittent systems.
 \textit{Comm. Math. Phys.} \textbf{230} (2002), no. 2, 365-388.
\bibitem[BV]{BV} Bomfim T and Varandas P 2017 Multifractal analysis for weak Gibbs measures: from large deviations to irregular sets. \textit{Ergodic Theory Dynam. Systems.} \textbf{37} 79-102.
\end{thebibliography}
\end{document}